\documentclass[a4paper,oneside,10pt]{article}%
\usepackage{amsmath}
\usepackage{amsfonts}
\usepackage{amssymb}
\usepackage{graphicx}
\usepackage{enumerate}
\usepackage{color}
\usepackage[square,numbers,sort&compress]{natbib}%
\usepackage{hyperref}
\usepackage[amsmath,thmmarks]{ntheorem}
\hypersetup{hypertex=true,bookmarks=true,bookmarksnumbered	=true,linktocpage=false,colorlinks=true,linkcolor=blue,anchorcolor=blue,citecolor=red}
\usepackage{indentfirst} 
\usepackage[mathscr]{euscript}
\providecommand{\U}[1]{\protect \rule{.1in}{.1in}}

\newtheorem{theorem}{Theorem}[section]

\newtheorem{corollary}[theorem]{Corollary}

\newtheorem{definition}[theorem]{Definition}

\newtheorem{lemma}[theorem]{Lemma}

\newenvironment{proof}[1][Proof]{\noindent \textbf{#1.} }{\  $\Box$}
\numberwithin{equation}{section}

\begin{document}

\makeatletter
\newcommand{\rmnum}[1]{\romannumeral #1}
\newcommand{\Rmnum}[1]{\expandafter\@slowromancap\romannumeral #1@}
\makeatother
{\theoremstyle{nonumberplain}
}

\title{$G$-BSDEs with time-varying monotonicity condition}

\author{Renxing Li\textsuperscript{1,}\thanks{Corresponding author. E-mail address: 202011963@mail.sdu.edu.cn} 
	\and Xue Zhang\textsuperscript{2}  
}

\footnotetext[1]{Zhongtai Securities Institute for Financial Studies,
	Shandong University, Jinan, Shandong 250100, PR China. 202011963@mail.sdu.edu.cn.}
\footnotetext[2]{Department of Mathematics, National University of Defense Technology, Changsha, Hunan 410073, PR China.
	zhangxue\_1998@nudt.edu.cn.}

\maketitle

\textbf{Abstract}. In this paper, we study backward stochastic differential equations driven by $G$-Brownian motion where the generator has time-varying monotonicity with respect to $y$ and Lipsitz property with respect to $z$. Through the Yosida approximation, we have proved the existence and uniqueness of the solutions to these equations.



\addcontentsline{toc}{section}{\hspace*{1.8em}Abstract}

\section{Introduction}

It is well-known that backward stochastic differential equations (BSDEs) play a crucial role in fields such as financial mathematics and stochastic control. In 1990, the existence and uniqueness of adapted solutions $(Y, Z)$ for BSDEs under Lipschitz conditions were proved by Pardoux and Peng\cite{pardoux1990adapted}, establishing the theoretical foundation for this field. Currently, extensive research and applications exist regarding BSDEs. Readers may refer to \cite{B2008BSDE,E1997backward,li2025weighted,liu2019multi,soner2012well,zhang2017backward,huy2000on} and their references.

In recent years, to address uncertainty issues in financial markets, Peng\cite{peng2004filtration,peng2007G,peng2008multi,peng2019nonlinear} introduced the $G$-expectation theory. Within the $G$-expectation framework, concepts such as $G$-Brownian motion and $G$-martingale were developed. Since these processes are no longer constructed based on a given probability space, they exhibit novel and interesting forms and properties, representing non-trivial generalizations of classical cases. Similarly, the new form of BSDEs driven by $G$-Brownian motion($G$-BSDEs) is given as follows:
\begin{equation}\label{G-BSDE}
	Y_{t}=\xi+\int_{t}^{T}f(s,Y_{s},Z_{s})ds+\int_{t}^{T}g(s,Y_{s},Z_{s})d\langle B\rangle _{s}-\int_{t}^{T}Z_{s}dB_{s}-(K_{T}-K_{t}).
\end{equation}
Here, $K$ denotes a continuous non-increasing $G$-martingale. Unlike the classical case, the solution to $G$-BSDEs is a triple $(Y, Z, K)$. 

Currently, numerous studies have been conducted on the existence and uniqueness of solutions for $G$-BSDEs. Hu\cite{hu2014backward,hu2024BSDE} investigated the existence and uniqueness of solutions to Equation (\ref{G-BSDE}) when the generator is uniformly continuous with respect to $(y, z)$ under both non-degenerate and degenerate cases. Bai and Lin\cite{bai2014on} studied the existence and uniqueness of solutions to Equation (\ref{G-BSDE}) under integral-Lipschitz conditions for generators. Research on generators satisfying time-varying Lipschitz conditions can be found in \cite{hu2020BSDEtimevary}. Quadratic $G$-BSDEs with convex generators and unbounded terminal conditions are investigated in \cite{huy2022quadratic}. Furthermore, some works have been made to relax generator assumptions to monotonicity conditions with respect to $y$, as detailed in \cite{song2019backward,wang2024reflectedBSDE} and their references.

In this paper, we study the $G$-BSDEs (\ref{G-BSDE}) where the generator has time-varying monotonicity with respect to $y$ and Lipsitz property with respect to $z$. In addition to the time-varying monotonicity with respect to $y$, we also need that generators $f,g$ are uniform continuity in $y$ and
\begin{equation}\label{grow}
	|\Phi(t,y,0)|\leq \phi(t)+u_{t}|y|, \ \Phi=f,g.
\end{equation}
Here $\phi(t) $ and $u_{t}$ are defined in (H3) of Section \ref{s3}. It is worth noting that since (\ref{grow}) is not a linear growth condition, the approximation method in \cite{song2019backward,wang2024reflectedBSDE} is not applicable here. 

Inspired by \cite{daprato1992stochastic,daprato1996ergodicity,huy2000on,pourrahimi2025stochastic}, we deal with continuous monotonic functions through the Yosida approximation. By Lemma \ref{falpha}, we know that the Yosida approximation not only transforms the time-varying monotonicity of generators with respect to $y$ into the time-varying Lipschitz property, but also preserves the Lipschitz property with respect to $z$.
By Lemma \ref{falphamg1}, we obtain that the approximate functions are in the $G$-expectation space. However, (iv) in Lemma \ref{falpha} represents pointwise convergence, which does not imply uniform convergence in the $G$-expectation space. In Lemma \ref{f-falpha}, we deduce that the approximate functions can converge uniformly to generators under the norm $\|\cdot \|_{M_{G}^{2}}$. Finally, by proving the convergence of the approximate Equation (\ref{GBSDEalpha}), we obtain the existence and uniqueness of the solution to Equation (\ref{GBSDEalpha}). Moreover, we also obtain the existence and uniqueness of the solution to Equation (\ref{G-BSDE}).

This paper is organized as follows. In Section \ref{s2}, we introduce some fundamental notations and results about $ G $-expectation theory. Section \ref{s3} presents the definition and properties of the Yosida approximation. The existence and uniqueness of the solution are provided in Section \ref{s4}.

\section{Preliminaries}\label{s2}

In this section, we recall the some basic notions and some necessary results of the $G$-expectation framework. The readers may refer to \cite{peng2004filtration,hu2016quasi,peng2007G,peng2008multi,peng2019nonlinear,denis2011function} for more details.

Let $\Omega =C_{0}(\mathbb{R}^{+})$ denote the space of all continuous functions $\omega: \mathbb{R}^{+}\rightarrow \mathbb{R}$ with $\omega _{0}=0$, equipped with the distance 
\begin{equation*}
	\rho (\omega^{1},\omega ^{2}):=\sum_{i=1}^{\infty }2^{-i}[(\max_{t\in
	\lbrack0,i]}|\omega_{t}^{1}-\omega_{t}^{2}|)\wedge1],\quad\omega^{1},\omega^{2}\in \Omega .
\end{equation*}%
Denote the canonical process by $B_{t}(\omega )=\omega _{t},t\in \mathbb{R}^{+}$. For each given $t\geq 0$, we define 
\begin{equation*}
	Lip(\Omega _{t}):=\{\varphi (B_{t_{1}\wedge t},\cdots ,B_{t_{n}\wedge t}):n\in \mathbb{N},t_{1},\cdots ,t_{n}\in \lbrack 0,\infty ),\varphi\in C_{l.Lip}(\mathbb{R}^{ n})\},\ Lip(\Omega ):=\cup _{t=1}^{\infty}Lip(\Omega _{t}),
\end{equation*}%
where $C_{l,Lip}(\mathbb{R}^{ n})$ is the space of all local Lipschitz functions on $\mathbb{R}^{ n}$.

Let function $G$ be defined as follows: for $ x\in\mathbb{R} $, $G(x)=\frac{1}{2}(\bar{\sigma}^{2}x^{+}-\underline{\sigma}^{2}x^{-}) $, where $ 0<\underline{\sigma}\leq\bar{\sigma}<\infty  $.
For each given function $ G$, Peng constructed the space $(\Omega ,Lip(\Omega ), \hat{\mathbb{E}}, (\hat{\mathbb{E}}_{t})_{t\geq0})$(see \cite{peng2019nonlinear} for definition). Then the canonical process $ B $ is a one-dimensional $ G $-Brownian motion under $ \hat{\mathbb{E}} $.

For each $p\geq 1$, $L_{G}^{p}(\Omega )$ is defined as the completion of $Lip(\Omega )$ under the norm $\Vert X\Vert_{L_{G}^{p}}=(\hat{\mathbb{E}}[|X|^{p}])^{\frac{1}{p}}$. For each $t>0$, $L_{G}^{p}(\Omega _{t})$ can be similarly defined. $(\Omega ,L_{G}^{1}(\Omega), \hat{\mathbb{E}})$ is called $ G $-expectation space.

\begin{theorem}[\cite{hu2009representation,denis2011function}]
	\label{thm2.1} Let $(\Omega ,L_{G}^{1}(\Omega),\hat{\mathbb{E}}) $ be a $G $-expectation space. Then there exists a weakly compact set of probability measures $\mathcal{P} $ on $(\Omega,\mathcal{F}) $ such that 
	\begin{align}
		\hat{\mathbb{E}}[\xi]=\sup_{P\in\mathcal{P}}E_{P}[\xi], \ \text{ for each }\xi\in L_{G}^{1}(\Omega).
	\end{align}
\end{theorem}

For this $\mathcal{P}$, we can define the capacity 
\begin{equation*}
	\mathrm{c}(A):=\sup_{P\in \mathcal{P}}P(A),\ \text{ for each }A\in \mathcal{F},
\end{equation*}
where $\mathcal{F}:=\bigvee_{t\geq 0}\mathcal{F}_{t}$ and $\mathcal{F}_{t}:=\sigma (B_{s}:s\leq t)$.

\begin{definition}
	A set $A\in \mathcal{F}$ is polar if $\mathrm{c}(A)=0$. Moreover, a property holds \textquotedblleft quasi-surely'' (q.s.) if it holds outside a polar set.
\end{definition}

For each $T>0$ and each $p\geq 1$, we define the following spaces:

\begin{itemize}
	\item $ M_{G}^{p,0}(0,T):=\left\{ \eta _{t}=\sum_{j=0}^{n-1}\xi
	_{j}I_{[t_{j},t_{j+1})}(t):n\in \mathbb{N},0=t_{0}<t_{1}<\cdots <t_{n}=T,\ \xi _{j}\in L_{G}^{p}(\Omega _{t_{j}})\right\} $;
	\item $ S_{G}^{0}(0,T):=\left\{ \psi(t,B_{t_{1}\wedge t},\cdots ,B_{t_{n}\wedge t}):t_{1},\cdots ,t_{n}\in[0,T],\ \psi \in C_{c,Lip}(\mathbb{R}^{n+1})\right\} $;
	\item $ M_{G}^{p}(0,T)$: the completion of $M_{G}^{p,0}(0,T)$ under the norm $\Vert\eta\Vert_{M_{G}^{p}}:=(\hat{\mathbb{E}}[\int_{0}^{T}|\eta
	_{t}|^{p}dt])^{\frac{1}{p}}$ ;
	\item $H_{G}^{p}(0,T)$: the completion of $M_{G}^{p,0}(0,T)$ under the norm $\Vert\eta\Vert_{H_{G}^{p}}:=(\hat{\mathbb{E}}[(\int_{0}^{T}|\eta
	_{t}|^{2}dt)^{\frac{p}{2}}])^{\frac{1}{p}}$;
	\item $S_{G}^{p}(0,T)$: the completion of $S_{G}^{0}(0,T)$ under the norm $\Vert\eta\Vert_{S_{G}^{p}}:=(\hat{\mathbb{E}}[\sup_{t\in \lbrack 0,T]}|\eta_{t}|^{p} ])^{\frac{1}{p}} $;
	\item $\mathfrak{S}_{G}^{p}(0,T):= \{ (Y,Z,K):  Y\in S_{G}^{p}(0,T),\ Z\in H_{G}^{p}(0,T),\ K \text{ is a non-increasing  $G$-martingale},\ K_{0} = 0,\ K_{T}\in L_{G}^{p}(\Omega_{T}) \}  $.
\end{itemize}

According to \cite{peng2019nonlinear}, the integrals $\int_{0}^{t}\eta_{s}\mathrm{d}B_{s}$ and $\int_{0}^{t}\mu_{s} \mathrm{d}\langle B\rangle _{s}$ are well-defined for $\eta \in H_{G}^{2}(0,T)$ and $\mu \in M_{G}^{1}(0,T)$, where $\langle B\rangle $ denotes the quadratic variation process of $B$. 

\begin{theorem}[\cite{hu2014backward}]
	Let $ p\geq1 $ and $ T>0 $. For all $ \eta\in H_{G}^{p}(0,T) $, we have
	\begin{align}\label{bdg} \underline{\sigma}^{p}c_{p}\hat{\mathbb{E}}\left[\left(\int_{0}^{T}|\eta_{s}|^{2}ds\right)^{\frac{p}{2}}\right]\leq\hat{\mathbb{E}}\left[ \sup_{t\in[0,T]}\left|\int_{0}^{t}\eta_{s}dB_{s} \right|^{p} \right]\leq\bar{\sigma}^{p}C_{p}\hat{\mathbb{E}}\left[\left(\int_{0}^{T}|\eta_{s}|^{2}ds \right)^{\frac{p}{2}}  \right], 
	 \end{align}
	where $ 0<c_{p}<C_{p}<\infty $ are constants.
\end{theorem}

In the following text, $C$ always denotes a positive number whose value depends on the subscript and may change from line to line.

\begin{theorem}[\cite{hu2014backward,hu2024BSDE,song2011some}]\label{EsupEt}
	For $q>p\geq 1$ and $\xi \in Lip(\Omega _{T})$, we have
	\begin{equation*}
		\hat{\mathbb{E}}\left[\sup_{t\in\lbrack0,T]}\hat{\mathbb{E}}_{t}\left[\left\vert \xi \right\vert ^{p}\right] \right] \leq C_{p,q}\left\{ \left( \hat{\mathbb{E}}\left[ \left\vert \xi \right\vert^{q}\right]\right)^{\frac{p}{q}}+\hat{\mathbb{E}}\left[ \left\vert \xi \right\vert ^{q}\right]\right\} .
	\end{equation*}
\end{theorem}

\section{Yosida approximation}\label{s3}
Set $T\in[0,\infty)$. For convenience, here we only study the $G$-BSDE in the following form:
\begin{align}\label{GBSDE}
	Y_{t}=\xi+\int_{t}^{T}f(s,Y_{s},Z_{s})ds-\int_{t}^{T}Z_{s}dB_{s}-(K_{T}-K_{t}).
\end{align}

Let $u_{t}:[0,T]\rightarrow\mathrm{R}^{+}$ be a process such that $\int_{0}^{T}u_{s}^{2}ds\leq M<\infty$. The generator $f(\omega,t,y,z):\Omega\times[0,T]\times\mathrm{R}\times\mathrm{R}\rightarrow\mathrm{R}$ satisfies the following assumptions.

\begin{description}
	\item[(H1)] For all $(y,z)\in \mathrm{R}^{2}$, $f(\cdot,\cdot,y,z)\in M_{G}^{1}(0,T)$. For each fixed $z$, $f$ is uniformly continuous in $y$. 
	
	\item[(H2)] $f$ satisfies the time-varying monotonicity condition with respect to $y$, i.e., 
	\[(y_{1}-y_{2})(f(t,y_{1},z)-f(t,y_{2},z))\leq u_{t}|y_{1}-y_{2}|^{2}. \]
	
	\item[(H3)] For constant $\lambda\geq0$, there exists a positive process $\phi$ such that 
	\[|f(t,y,0)|\leq \phi(t)+u_{t}|y|,\] 
	and $|f(\cdot,0,0)|\vee \phi := h\in M_{G}^{2+\lambda}(0,T)$.
	
	\item[(H4)] $f$ satisfies the Lipschitz condition with respect to $z$, i.e., 
	\[|f(t,y_{1},z)-f(t,y_{2},z)|\leq L|z_{1}-z_{2}|. \]
\end{description}

In the following, we approximate $f$ by constructing a sequence of functions via Yosida approximation.

\begin{lemma}\label{Falpha}
	Let $f$ satisfy (H1)-(H2), (H4) and let $F(\omega,t,y,z)=f(\omega,t,y,z)-u_{t}y$. For any $\alpha>0$ and  $\omega,t,y,z$, the following equation exists a unique solution $x$,
	\[x-\alpha F(\omega,t,x,z)=y. \]
	Define $x=J^{\alpha}(\omega,t,y,z) $ and \[F^{\alpha}(\omega,t,y,z)=F(\omega,t,J^{\alpha}(\omega,t,y,z),z)=\frac{1}{\alpha}(J^{\alpha}(\omega,t,y,z)-y).\]
	Then $F^{\alpha}(\omega,t,y,z)$ satisfy the following properties.
	
	\begin{description}
		\item[(i)] For $y_{1},y_{2}\in\mathrm{R}$, 
		\begin{align*}
			&|J^{\alpha}(\omega,t,y_{1},z)-J^{\alpha}(\omega,t,y_{2},z)|\leq|y_{1}-y_{2}|,\\
			&|F^{\alpha}(\omega,t,y_{1},z)-F^{\alpha}(\omega,t,y_{2},z)|\leq\frac{2}{\alpha}|y_{1}-y_{2}|,\\
			&\left( F^{\alpha}(\omega,t,y_{1},z)-F^{\alpha}(\omega,t,y_{2},z)\right) (y_{1}-y_{2})\leq0.
		\end{align*}
		
		\item[(ii)] \[|F^{\alpha}(\omega,t,y,z)|\leq|F(\omega,t,y,z)|.\]
		
		\item[(iii)]
		\[\lim\limits_{\alpha\rightarrow0}J^{\alpha}(\omega,t,y,z)=y.\]
		
		\item[(iv)] For $\alpha,\beta>0$ and $y_{1},y_{2}\in\mathrm{R}$,
		\[\left( F^{\alpha}(\omega,t,y_{1},z)-F^{\beta}(\omega,t,y_{2},z)\right) (y_{1}-y_{2})\leq(\alpha+\beta)\left( |F^{\alpha}(\omega,t,y_{1},z)|+|F^{\beta}(\omega,t,y_{2},z)|\right)^{2}.\]
		
		\item[(v)] If there exists $\{y^{\alpha}\}$ such that 
		\[\lim\limits_{\alpha\rightarrow0}y^{\alpha}=y, \]
		then 
		\[ \lim\limits_{\alpha\rightarrow0}F^{\alpha}(\omega,t,y^{\alpha},z)=F(\omega,t,y,z). \]	
		
		\item[(vi)]	For $z_{1},z_{2}\in\mathrm{R}$, 
		\begin{align*}
			&|J^{\alpha}(\omega,t,y,z_{1})-J^{\alpha}(\omega,t,y,z_{2})|\leq\alpha L|z_{1}-z_{2}|,\\
			&|F^{\alpha}(\omega,t,y,z_{1})-F^{\alpha}(\omega,t,y,z_{2})|\leq L|z_{1}-z_{2}|.
		\end{align*}
		
	\end{description}	
\end{lemma}

\begin{proof}
	According to (H2), it is clearly that 
	\[\left( F(\omega,t,y_{1},z)-F(\omega,t,y_{2},z)\right) (y_{1}-y_{2})\leq0.  \]
	This indicates that $F$ is a continuous dissipative mapping. From Corollary D.10 in \cite{daprato1992stochastic}, the solution $J^{\alpha}(\omega,t,y,z) $ exists and is unique.
	
	(i): It is obvious that
	\begin{equation}\label{pf(i)}
		\left(J^{\alpha}(\omega,t,y_{1},z)-J^{\alpha}(\omega,t,y_{2},z) \right)-\alpha\left(F(\omega,t,J^{\alpha}(\omega,t,y_{1},z),z)-F(\omega,t,J^{\alpha}(\omega,t,y_{2},z),z) \right)=y_{1}-y_{2}.
	\end{equation}
	Multiply both sides of the equation by $\left(J^{\alpha}(\omega,t,y_{1},z)-J^{\alpha}(\omega,t,y_{2},z) \right)$, then
	\begin{align*}
		&\left(J^{\alpha}(\omega,t,y_{1},z)-J^{\alpha}(\omega,t,y_{2},z) \right)^{2}-(y_{1}-y_{2})\left(J^{\alpha}(\omega,t,y_{1},z)-J^{\alpha}(\omega,t,y_{2},z) \right)\\
		=&\alpha\left(F(\omega,t,J^{\alpha}(\omega,t,y_{1},z),z)-F(\omega,t,J^{\alpha}(\omega,t,y_{2},z),z) \right)\left(J^{\alpha}(\omega,t,y_{1},z)-J^{\alpha}(\omega,t,y_{2},z) \right).
	\end{align*}
	Since 
	\[\alpha\left(F(\omega,t,J^{\alpha}(\omega,t,y_{1},z),z)-F(\omega,t,J^{\alpha}(\omega,t,y_{2},z),z) \right)\left(J^{\alpha}(\omega,t,y_{1},z)-J^{\alpha}(\omega,t,y_{2},z) \right)\leq0,\]
	we have $(y_{1}-y_{2})\left(J^{\alpha}(\omega,t,y_{1},z)-J^{\alpha}(\omega,t,y_{2},z) \right)\geq0 $.
	Thus, 
	\[\left(J^{\alpha}(\omega,t,y_{1},z)-J^{\alpha}(\omega,t,y_{2},z) \right)^{2}\leq (y_{1}-y_{2})\left(J^{\alpha}(\omega,t,y_{1},z)-J^{\alpha}(\omega,t,y_{2},z) \right).  \]
	It follows that $|J^{\alpha}(\omega,t,y_{1},z)-J^{\alpha}(\omega,t,y_{2},z)|\leq|y_{1}-y_{2}|$. 
	
	Then, according to the Equation (\ref{pf(i)}), we get
	\begin{align*}
		&\left|F(\omega,t,J^{\alpha}(\omega,t,y_{1},z),z)-F(\omega,t,J^{\alpha}(\omega,t,y_{2},z),z) \right|\\
		=&\frac{1}{\alpha}\left|J^{\alpha}(\omega,t,y_{1},z)-J^{\alpha}(\omega,t,y_{2},z)-(y_{1}-y_{2}) \right|\\
		\leq&\frac{2}{\alpha}|y_{1}-y_{2}|. 
	\end{align*}
	
	Next, multiply both sides of Equation (\ref{pf(i)}) by $(y_{1}-y_{2})$, we obtain
	\begin{align*}
		&\left(J^{\alpha}(\omega,t,y_{1},z)-J^{\alpha}(\omega,t,y_{2},z) \right)(y_{1}-y_{2})-(y_{1}-y_{2})^{2}\\
		=&\alpha\left(F(\omega,t,J^{\alpha}(\omega,t,y_{1},z),z)-F(\omega,t,J^{\alpha}(\omega,t,y_{2},z),z) \right)(y_{1}-y_{2}).
	\end{align*}
	Recall that $(y_{1}-y_{2})\left(J^{\alpha}(\omega,t,y_{1},z)-J^{\alpha}(\omega,t,y_{2},z) \right)\geq0 $ and $|J^{\alpha}(\omega,t,y_{1},z)-J^{\alpha}(\omega,t,y_{2},z)|\leq|y_{1}-y_{2}| $, it follows that
	\[\alpha\left(F(\omega,t,J^{\alpha}(\omega,t,y_{1},z),z)-F(\omega,t,J^{\alpha}(\omega,t,y_{2},z),z) \right)(y_{1}-y_{2})\leq0.\]
	
	(ii) Subtract $y-\alpha F(\omega,t,y,z)$ from both sides of the equation $J^{\alpha}(\omega,t,y,z)-\alpha F(\omega,t,J^{\alpha}(\omega,t,y,z),z)=y$ simultaneously, then we yield 
	\begin{align*}
		J^{\alpha}(\omega,t,y,z)-y-\alpha\left(F(\omega,t,J^{\alpha}(\omega,t,y,z),z)-F(\omega,t,y,z) \right) =\alpha F(\omega,t,y,z).
	\end{align*}
	Then multiply both sides of the equation by $\left(F(\omega,t,J^{\alpha}(\omega,t,y,z),z)-F(\omega,t,y,z)  \right)$, 
	\begin{align*}
		&\alpha\left(F(\omega,t,J^{\alpha}(\omega,t,y,z),z)-F(\omega,t,y,z) \right)^{2}-\left(J^{\alpha}(\omega,t,y,z)-y \right) \left(F(\omega,t,J^{\alpha}(\omega,t,y,z),z)-F(\omega,t,y,z)  \right)\\
		=&-\alpha F(\omega,t,y,z) \left(F(\omega,t,J^{\alpha}(\omega,t,y,z),z)-F(\omega,t,y,z)  \right).
	\end{align*}
	Since $-\left(J^{\alpha}(\omega,t,y,z)-y \right) \left(F(\omega,t,J^{\alpha}(\omega,t,y,z),z)-F(\omega,t,y,z)  \right)\geq0 $, it follows that 
	\[-\alpha F(\omega,t,y,z) \left(F(\omega,t,J^{\alpha}(\omega,t,y,z),z)-F(\omega,t,y,z)  \right)\geq0. \]
	Thus 
	\[ \left(F(\omega,t,J^{\alpha}(\omega,t,y,z),z)-F(\omega,t,y,z) \right)^{2}\leq- F(\omega,t,y,z) \left(F(\omega,t,J^{\alpha}(\omega,t,y,z),z)-F(\omega,t,y,z)  \right) . \]
	Through standard calculation, $|F(\omega,t,J^{\alpha}(\omega,t,y,z),z)|^{2}\leq|F(\omega,t,J^{\alpha}(\omega,t,y,z),z)F(\omega,t,y,z)|$ can be obtained. Then we get the desire result.
	
	(iii) From (ii), we get 
	\begin{align*}
		|y-J^{\alpha}(\omega,t,y,z)|=&|\alpha F^{\alpha}(\omega,t,y,z)|\\
		\leq & \alpha |F(\omega,t,y,z)|.
	\end{align*}
	Thus, $\lim\limits_{\alpha\rightarrow0}|y-J^{\alpha}(\omega,t,y,z)|=0$.
	
	(iv) It is easy to check that
	\[J^{\alpha}(\omega,t,y_{1},z)-J^{\beta}(\omega,t,y_{2},z)-(\alpha F^{\alpha}(\omega,t,y_{1},z)-\beta F^{\beta}(\omega,t,y_{2},z))=y_{1}-y_{2}. \]
	Then 
	\begin{align*}
		&\left( F^{\alpha}(\omega,t,y_{1},z)-F^{\beta}(\omega,t,y_{2},z)\right)(y_{1}-y_{2})\\
		=&\left( F^{\alpha}(\omega,t,y_{1},z)-F^{\beta}(\omega,t,y_{2},z)\right) \left( J^{\alpha}(\omega,t,y_{1},z)-J^{\beta}(\omega,t,y_{2},z) \right) \\
		&-\left( F^{\alpha}(\omega,t,y_{1},z)-F^{\beta}(\omega,t,y_{2},z)\right) \left( \alpha F^{\alpha}(\omega,t,y_{1},z)-\beta F^{\beta}(\omega,t,y_{2},z) \right).
	\end{align*}
	By the definition of $F^{\alpha}$, it follows that $\left( F^{\alpha}(\omega,t,y_{1},z)-F^{\beta}(\omega,t,y_{2},z)\right) \left( J^{\alpha}(\omega,t,y_{1},z)-J^{\beta}(\omega,t,y_{2},z) \right)\leq0 $. Now we have 
	\begin{align*}
		&\left( F^{\alpha}(\omega,t,y_{1},z)-F^{\beta}(\omega,t,y_{2},z)\right)(y_{1}-y_{2})\\
		\leq&-\left( F^{\alpha}(\omega,t,y_{1},z)-F^{\beta}(\omega,t,y_{2},z)\right) \left( \alpha F^{\alpha}(\omega,t,y_{1},z)-\beta F^{\beta}(\omega,t,y_{2},z) \right)\\
		\leq&-\alpha\left(F^{\alpha}(\omega,t,y_{1},z) \right)^{2}+(\alpha+\beta)F^{\alpha}(\omega,t,y_{1},z)F^{\beta}(\omega,t,y_{2},z) -\beta\left(F^{\alpha}(\omega,t,y_{2},z) \right)^{2}\\
		\leq& (\alpha+\beta)\left( |F^{\alpha}(\omega,t,y_{1},z)|+|F^{\beta}(\omega,t,y_{2},z)| \right)^{2}.
	\end{align*}
	
	(v) Note that 
	\[ F^{\alpha}(\omega,t,y^{\alpha},z)=F(\omega,t,J^{\alpha}(\omega,t,y^{\alpha},z),z). \]
	From (i) and (iii), we get
	\begin{align*}
		&\lim\limits_{\alpha\rightarrow0}|J^{\alpha}(\omega,t,y^{\alpha},z)-y|\\
		\leq& \lim\limits_{\alpha\rightarrow0}|J^{\alpha}(\omega,t,y^{\alpha},z)-J^{\alpha}(\omega,t,y,z)|+\lim\limits_{\alpha\rightarrow0}|J^{\alpha}(\omega,t,y,z)-y|\\
		\leq&\lim\limits_{\alpha\rightarrow0}|y^{\alpha}-y|+\lim\limits_{\alpha\rightarrow0}|J^{\alpha}(\omega,t,y,z)-y|\\
		=&0.
	\end{align*}
	According to the continuity of $F$, we obtain the desired result.
	
	(vi) Clearly,
	\begin{equation}\label{pf(vi)}
		J^{\alpha}(\omega,t,y,z_{1})-J^{\alpha}(\omega,t,y,z_{2})-\alpha\left( F(\omega,t,J^{\alpha}(\omega,t,y,z_{1}),z_{1})-F(\omega,t,J^{\alpha}(\omega,t,y,z_{2}),z_{2})  \right)=0.
	\end{equation}
	It follows that 
	\begin{align*}
		&J^{\alpha}(\omega,t,y,z_{1})-J^{\alpha}(\omega,t,y,z_{2})-\alpha\left( F(\omega,t,J^{\alpha}(\omega,t,y,z_{1}),z_{1})-F(\omega,t,J^{\alpha}(\omega,t,y,z_{2}),z_{1})  \right)\\
		=&\alpha\left( F(\omega,t,J^{\alpha}(\omega,t,y,z_{2}),z_{1})-F(\omega,t,J^{\alpha}(\omega,t,y,z_{2}),z_{2})  \right).
	\end{align*}
	Then multiply both sides of the equation by $J^{\alpha}(\omega,t,y,z_{1})-J^{\alpha}(\omega,t,y,z_{2}) $, we yield
	\begin{align*}
		&\left( J^{\alpha}(\omega,t,y,z_{1})-J^{\alpha}(\omega,t,y,z_{2})\right) ^{2}\\
		=&\alpha\left( F(\omega,t,J^{\alpha}(\omega,t,y,z_{1}),z_{1})-F(\omega,t,J^{\alpha}(\omega,t,y,z_{2}),z_{1})  \right)\left( J^{\alpha}(\omega,t,y,z_{1})-J^{\alpha}(\omega,t,y,z_{2})\right)\\
		&+\alpha\left( F(\omega,t,J^{\alpha}(\omega,t,y,z_{2}),z_{1})-F(\omega,t,J^{\alpha}(\omega,t,y,z_{2}),z_{2})  \right)\left( J^{\alpha}(\omega,t,y,z_{1})-J^{\alpha}(\omega,t,y,z_{2})\right).
	\end{align*}
	Due to $\alpha\left( F(\omega,t,J^{\alpha}(\omega,t,y,z_{1}),z_{1})-F(\omega,t,J^{\alpha}(\omega,t,y,z_{2}),z_{1})  \right)\left( J^{\alpha}(\omega,t,y,z_{1})-J^{\alpha}(\omega,t,y,z_{2})\right)\leq0 $, we have
	\begin{align*}
		&\left( J^{\alpha}(\omega,t,y,z_{1})-J^{\alpha}(\omega,t,y,z_{2})\right) ^{2}\\
		\leq&\alpha\left( F(\omega,t,J^{\alpha}(\omega,t,y,z_{2}),z_{1})-F(\omega,t,J^{\alpha}(\omega,t,y,z_{2}),z_{2})  \right)\left( J^{\alpha}(\omega,t,y,z_{1})-J^{\alpha}(\omega,t,y,z_{2})\right).
	\end{align*}
	Combining (H4), we obtain
	\begin{align*}
		&\left| J^{\alpha}(\omega,t,y,z_{1})-J^{\alpha}(\omega,t,y,z_{2})\right|\\
		\leq&\left|\alpha F(\omega,t,J^{\alpha}(\omega,t,y,z_{2}),z_{1})-F(\omega,t,J^{\alpha}(\omega,t,y,z_{2}),z_{2}) \right| \\
		\leq&\alpha\left| f(\omega,t,J^{\alpha}(\omega,t,y,z_{2}),z_{1})-f(\omega,t,J^{\alpha}(\omega,t,y,z_{2}),z_{2}) \right|  \\
		\leq& \alpha L|z_{1}-z_{2}|.
	\end{align*}
	
	Then by Equation (\ref{pf(vi)}), we yield
	\begin{align*}
		&\left| F(\omega,t,J^{\alpha}(\omega,t,y,z_{1}),z_{1})-F(\omega,t,J^{\alpha}(\omega,t,y,z_{2}),z_{2})  \right| \\
		=&\frac{1}{\alpha}\left|J^{\alpha}(\omega,t,y,z_{1})-J^{\alpha}(\omega,t,y,z_{2}) \right| \\
		\leq&L|z_{1}-z_{2}|.
	\end{align*}
	The proof is completed.
\end{proof}

\begin{lemma}\label{falpha}
	Let $f$ satisfy (H1)-(H4) and let $F, F^{\alpha},J^{\alpha}$ be defined as in Lemma \ref{Falpha}. Define 
	\[f^{\alpha}(\omega,t,y,z)=F^{\alpha}(\omega,t,y,z)+u_{t}y. \]
	Then $f^{\alpha}(\omega,t,y,z)$ satisfy the following properties.
	\begin{description}
		\item[(i)] For $y_{1},y_{2}\in\mathrm{R}$, 
		\begin{align*}
			&|f^{\alpha}(\omega,t,y_{1},z)-f^{\alpha}(\omega,t,y_{2},z)|\leq(\frac{2}{\alpha}+u_{t})|y_{1}-y_{2}|,\\
			&\left( f^{\alpha}(\omega,t,y_{1},z)-f^{\alpha}(\omega,t,y_{2},z)\right)(y_{1}-y_{2}) \leq u_{t}|y_{1}-y_{2}|^{2}.
		\end{align*}
		
		\item[(ii)]
		\begin{align*}
			&|f^{\alpha}(\omega,t,y,z)|\leq|f(\omega,t,y,z)|+2u_{t}|y|,\\
			&|f^{\alpha}(\omega,t,y,0)|\leq h(t)+3u_{t}|y|.
		\end{align*}
		
		\item[(iii)] For $\alpha,\beta>0$ and $y_{1},y_{2}\in\mathrm{R}$,
		\begin{align*}
			&\left( f^{\alpha}(\omega,t,y_{1},z)-f^{\beta}(\omega,t,y_{2},z)\right) (y_{1}-y_{2})\\
			\leq&(\alpha+\beta)\left( |f^{\alpha}(\omega,t,y_{1},z)|+|f^{\beta}(\omega,t,y_{2},z)|+u_{t}(|y_{1}|+|y_{2}|)\right)^{2}+u_{t}|y_{1}-y_{2}|^{2}
		\end{align*}
				
		\item[(iv)] If there exists $\{y^{\alpha}\}$ such that 
		\[\lim\limits_{\alpha\rightarrow0}y^{\alpha}=y, \]
		then 
		\[ \lim\limits_{\alpha\rightarrow0}f^{\alpha}(\omega,t,y^{\alpha},z)=f(\omega,t,y,z). \]	
		
		\item[(v)] For $z_{1},z_{2}\in\mathrm{R}$, 
		\[|f^{\alpha}(\omega,t,y,z_{1})-f^{\alpha}(\omega,t,y,z_{2})|\leq L|z_{1}-z_{2}|. \]
	\end{description}	
	
\end{lemma}

\begin{proof}
	According to Lemma \ref{Falpha}, it is easy to prove (i)-(v).
\end{proof}

Now, by the Yosida approximation, we construct a series of functions $f^{\alpha}$ that satisfy the Lipschitz property with respect to both $y$ and $z$. In the following, we need to verify that $f^{\alpha}$ is well-defined under the $G$-expectation.

\begin{lemma}\label{falphamg1}
	Let $f$ satisfy (H1)-(H4) and let $f^{\alpha}$ be defined as in Lemma \ref{falpha}. Then $f^{\alpha}(\cdot,\cdot,y,z)\in M_{G}^{1}(0,T)$.
\end{lemma}

\begin{proof}
	In view of the definition of $f^{\alpha}$, we only need to prove that $F^{\alpha}(\cdot,\cdot,y,z)\in M_{G}^{1}(0,T)$.
	
	First, we need to verify the quasi-continuity of $F^{\alpha}$. For each fixed $t,y,z$ and $\alpha$, we denote $m(\omega,t):=J^{\alpha}(\omega,t,y,z)$ for convenience. For $(\omega_{1},t_{1}),(\omega_{2},t_{2})\in\Omega\times [0,T]$, we have
	\begin{align*}
		&m(\omega_{1},t_{1})-\alpha F(\omega_{1},t_{1},m(\omega_{1},t_{1}),z)=y,\\
		&m(\omega_{2},t_{2})-\alpha F(\omega_{2},t_{2},m(\omega_{2},t_{2}),z)=y.
	\end{align*}
	Taking the difference between the above two equations, we get
	\[m(\omega_{1},t_{1})-m(\omega_{2},t_{2})-\alpha\left(F(\omega_{1},t_{1},m(\omega_{1},t_{1}),z)-F(\omega_{2},t_{2},m(\omega_{2},t_{2}),z) \right)=0.   \]
	It follows that 
	\begin{align*}
		&{}m(\omega_{1},t_{1})-m(\omega_{2},t_{2})-\alpha\left(F(\omega_{1},t_{1},m(\omega_{1},t_{1}),z)-F(\omega_{1},t_{1},m(\omega_{2},t_{2}),z) \right)\\
		={}&\alpha\left(F(\omega_{1},t_{1},m(\omega_{2},t_{2}),z)-F(\omega_{2},t_{2},m(\omega_{2},t_{2}),z) \right)
	\end{align*}
	Then multiply both sides of the equation by $m(\omega_{1},t_{1})-m(\omega_{2},t_{2})$, we yield
	\begin{align*}
		&\left(m(\omega_{1},t_{1})-m(\omega_{2},t_{2})\right) ^{2}-\alpha\left(F(\omega_{1},t_{1},m(\omega_{1},t_{1}),z)-F(\omega_{1},t_{1},m(\omega_{2},t_{2}),z) \right)\left( m(\omega_{1},t_{1})-m(\omega_{2},t_{2})\right)\\
		=&\alpha\left(F(\omega_{1},t_{1},m(\omega_{2},t_{2}),z)-F(\omega_{2},t_{2},m(\omega_{2},t_{2}),z) \right)\left( m(\omega_{1},t_{1})-m(\omega_{2},t_{2})\right).
	\end{align*}
	Notice that $-\alpha\left(F(\omega_{1},t_{1},m(\omega_{1},t_{1}),z)-F(\omega_{1},t_{1},m(\omega_{2},t_{2}),z) \right)\left( m(\omega_{1},t_{1})-m(\omega_{2},t_{2})\right)\geq0 $.
	Thus, we can obtain
	\[\left|m(\omega_{1},t_{1})-m(\omega_{2},t_{2}) \right|\leq\alpha\left| F(\omega_{1},t_{1},m(\omega_{2},t_{2}),z)-F(\omega_{2},t_{2},m(\omega_{2},t_{2}),z) \right|.  \]
	According to Theorem 4.7 in \cite{hu2016quasi}, we know that $F$ is quasi-continuous with respect to $(\omega,t)$ in the sense of capacity  $\hat{c}$. Then we yield $m$ is quasi-continuous with respect to $(\omega,t)$ in the sense of capacity $\hat{c}$. From (H1), we get $F$ is continuous in $y$. Thus,
	\begin{align*}
		&\left| F^{\alpha}(\omega_{1},t_{1},y,z)-F^{\alpha}(\omega_{2},t_{2},y,z)\right|\\ 
		\leq&\left| F(\omega_{1},t_{1},m(\omega_{1},t_{1}),z)-F(\omega_{2},t_{2},m(\omega_{2},t_{2}),z)\right|\\
		\leq&\left| F(\omega_{1},t_{1},m(\omega_{1},t_{1}),z)-F(\omega_{1},t_{1},m(\omega_{2},t_{2}),z)\right|+\left| F(\omega_{1},t_{1},m(\omega_{2},t_{2}),z)-F(\omega_{2},t_{2},m(\omega_{2},t_{2}),z)\right|.
	\end{align*}
	It follows that 
	\[\left| F^{\alpha}(\omega_{1},t_{1},y,z)-F^{\alpha}(\omega_{2},t_{2},y,z)\right|\rightarrow0 \text{  as  } (\omega_{1},t_{1})\rightarrow (\omega_{2},t_{2}).\]
	Therefore, $F^{\alpha}$ is quasi-continuous with respect to $(\omega,t)$ in the sense of capacity  $\hat{c}$.
	 
	Finally, we need to prove the integrability of $F^{\alpha}$. According to (ii) in Lemma \ref{Falpha}, it leads to 
	\begin{equation*}
		\hat{\mathbb{E}}\left[ \int_{0}^{T}\left\vert F^{\alpha }(\omega,s,y,z)\right\vert I_{\left\lbrace\left\vert F^{\alpha }(\omega,s,y,z)\right\vert\geq N \right\rbrace }ds\right]
		\leq \hat{\mathbb{E}}\left[ \int_{0}^{T}\left\vert F(\omega,s,y,z)\right\vert I_{\left\lbrace\left\vert F(\omega,s,y,z)\right\vert\geq N \right\rbrace }ds\right].
	\end{equation*}
	Since $F(\cdot,\cdot,y,z)\in M_{G}^{1}(0,T)$, it follows that $\lim\limits_{N \rightarrow \infty} \hat{\mathbb{E}}\left[ \int_{0}^{T}\left\vert F^{\alpha }(\omega,s,y,z)\right\vert I_{\left\lbrace\left\vert F^{\alpha }(\omega,s,y,z)\right\vert\geq N \right\rbrace }ds\right]=0$. This completes the proof.
\end{proof}

For the rest of this paper, we will omit the $\omega$ term for convenience.

\begin{lemma}\label{f-falpha}
	Let $f$ satisfy (H1)-(H4) and let $f^{\alpha }$ be defined as in Lemma \ref{falpha}. If $\left( Y,Z\right) \in S_{G}^{2}(0,T)\times H_{G}^{2}(0,T)$, then $f^{\alpha }(\cdot ,Y,Z),f(\cdot ,Y,Z)\in M_{G}^{2}(0,T)$ and 
	\begin{equation}
		\lim\limits_{\alpha \rightarrow 0}\hat{\mathbb{E}}\left[ \int_{0}^{T}\left\vert f(s,Y_{s},Z_{s})-f^{\alpha }(s,Y_{s},Z_{s})\right\vert ^{2}ds\right]=0.  \label{fMG2}
	\end{equation}
\end{lemma}

\begin{proof}
	According to Lemma 3.2 in \cite{hu2020BSDEtimevary}, we know $f^{\alpha}(\cdot ,Y,Z)\in M_{G}^{1}(0,T)$. From Lemma \ref{falpha}, we get
	\begin{align*}
		&\hat{\mathbb{E}}\left[ \int_{0}^{T}\left\vert f^{\alpha
		}(s,Y_{s},Z_{s})\right\vert ^{2}ds\right]  \\
		\leq &\hat{\mathbb{E}}\left[ \int_{0}^{T}\left\vert h(s)+3u_{s}\left\vert Y_{s}\right\vert +L\left\vert Z_{s}\right\vert \right\vert ^{2}ds\right]  \\
		\leq &C_{L}\left( \hat{\mathbb{E}}\left[ \int_{0}^{T}\left\vert
		h(s)\right\vert ^{2}ds\right] +\hat{\mathbb{E}}\left[ \int_{0}^{T}u_{s}^{2}\left\vert Y_{s}\right\vert ^{2}ds\right] +\hat{\mathbb{E}}\left[\int_{0}^{T}\left\vert Z_{s}\right\vert ^{2}ds\right] \right)  \\
		\leq &C_{T,L,M}\left( \hat{\mathbb{E}}\left[ \int_{0}^{T}\left\vert h(s)\right\vert ^{2}ds\right] +\left\Vert Y\right\Vert_{S_{G}^{2}}^{2}+\left\Vert Z\right\Vert _{H_{G}^{2}}^{2}\right) .
	\end{align*}%
	Thus, $f^{\alpha }(\cdot ,Y,Z)\in M_{G}^{2}(0,T)$. Clearly, $F^{\alpha }(\cdot ,Y,Z)\in M_{G}^{2}(0,T)$. It is easy to check that  
	\begin{equation*}
		F(t,Y_{t},Z_{t})\leq h_{t}+2u_{t}|Y_{t}|+L|Z_{t}|.
	\end{equation*}
	Since $F$ is continuous with respect to $y$, this indicates that for any $\epsilon >0$, there exists $\delta $ that satisfies: if $\left\vert y-y^{\prime }\right\vert <\delta $, then $|F(t,y,z)-F(t,y^{\prime },z)|<\epsilon $. Let $A_{\alpha }=\left\{
	\left\vert Y_{s}-J^{\alpha }(s,Y_{s},Z_{s})\right\vert <\delta \right\} $. From the definition of $J^{\alpha }(s,Y_{s},Z_{s})$ and Lemma \ref{Falpha}, we yield 
	\[A_{\alpha }^{c}=\left\{ \alpha \left\vert F^{\alpha}(s,Y_{s},Z_{s})\right\vert \geq \delta \right\} \subset B_{\alpha }=\left\{\alpha \left( h_{s}+2u_{s}|Y_{s}|+L|Z_{s}| \right) \geq \delta \right\} .\]
	Clearly, 
	\begin{align*}
		\lim _{\alpha \rightarrow 0}\hat{\mathbb{E}}\left[ \int_{0}^{T}I_{B_{\alpha }}\mathrm{d}s\right] \leq{}& \lim _{\alpha \rightarrow 0}\alpha ^{2}\hat{\mathbb{E}}\left[ \int_{0}^{T}\frac{\left( h_{s}+2u_{s}|Y_{s}|+L|Z_{s}|\right) ^{2}}{\delta ^{2}}\mathrm{d}s\right] \\
		\leq{}&\lim _{\alpha \rightarrow 0}\frac{\alpha ^{2}}{\delta^{2}}C_{T,L,M}\left( \hat{\mathbb{E}}\left[ \int_{0}^{T}\left\vert h(s)\right\vert ^{2}\mathrm{d}s\right] +\left\Vert Y\right\Vert_{S_{G}^{2}}^{2}+\left\Vert Z\right\Vert _{H_{G}^{2}}^{2}\right)\\
		={}&0.
	\end{align*}
	By Lemma 4.2 in \cite{hu2016stochastic}, we know that 
	\begin{equation*}
		\lim _{\alpha \rightarrow 0}\hat{\mathbb{E}}\left[ \int_{0}^{T}\left( h_{s}+2u_{s}|Y_{s}|+L|Z_{s}| \right) ^{2}I_{B_{\alpha }}\mathrm{d}s\right] <\epsilon ,
	\end{equation*}
	where $\epsilon $ is the same as above. Due to Lemma \ref{Falpha},   
	\begin{align*}
		&{}\lim _{\alpha \rightarrow 0}\hat{\mathbb{E}}\left[
		\int_{0}^{T}\left\vert f(s,Y_{s},Z_{s})-f^{\alpha
		}(s,Y_{s},Z_{s})\right\vert ^{2}\mathrm{d}s\right]  \\
		\leq{} &\lim _{\alpha \rightarrow 0}\hat{\mathbb{E}}\left[
		\int_{0}^{T}\left\vert F(s,Y_{s},Z_{s})-F^{\alpha
		}(s,Y_{s},Z_{s})\right\vert ^{2}\mathrm{d}s\right]  \\
		\leq{} &\lim _{\alpha \rightarrow 0}\hat{\mathbb{E}}\left[
		\int_{0}^{T}\left\vert F(s,Y_{s},Z_{s})-F^{\alpha
		}(s,Y_{s},Z_{s})\right\vert ^{2}\left( I_{A_{\alpha }}+I_{A_{\alpha}^{c}}\right) \mathrm{d}s\right]  \\
		\leq{} &T\epsilon ^{2}+\lim _{\alpha \rightarrow 0}\hat{\mathbb{E}}\left[ \int_{0}^{T}\left\vert F(s,Y_{s},Z_{s})-F^{\alpha
		}(s,Y_{s},Z_{s})\right\vert ^{2}I_{B_{\alpha }}\mathrm{d}s\right]  \\
		\leq{} & T\epsilon ^{2}+4\lim _{\alpha \rightarrow 0}\hat{\mathbb{E}}\left[ \int_{0}^{T}\left\vert F(s,Y_{s},Z_{s})\right\vert ^{2}I_{B_{\alpha}}\mathrm{d}s\right]  \\
		\leq{} & T\epsilon ^{2}+4\lim _{\alpha \rightarrow 0}\hat{\mathbb{E}}\left[ \int_{0}^{T}\left( h_{s}+2u_{s}|Y_{s}|+L|Z_{s}| \right) ^{2}I_{B_{\alpha}}\mathrm{d}s\right]\\
		\leq{} &T\epsilon ^{2}+4\epsilon .
	\end{align*}%
	Let $\epsilon $ approach 0, we get (\ref{fMG2}). Therefore, $f(\cdot
	,Y,Z)\in M_{G}^{2}(0,T)$.
\end{proof}

\section{Existence and uniqueness of the solution}\label{s4}

For any fixed $\alpha>0$, we consider the $G$-BSDE with generator $f^{\alpha}$:
\begin{align}\label{GBSDEalpha}
	Y_{t}^{\alpha}=\xi+\int_{t}^{T}f^{\alpha}(s,Y_{s}^{\alpha},Z_{s}^{\alpha})ds-\int_{t}^{T}Z_{s}^{\alpha}dBs-(K_{T}^{\alpha}-K_{t}^{\alpha}),
\end{align}
where $f^{\alpha}$ is defined in Lemma \ref{falpha}. According to Theorem 3.3 in \cite{hu2020BSDEtimevary}, we know that when $\lambda>0$, if $\xi\in L_{G}^{2+\lambda}$ and $f$ satisfies (H1)-(H4), then $G$-BSDE (\ref{GBSDEalpha}) has a unique solution $(Y^{\alpha},Z^{\alpha},K^{\alpha})\in \mathfrak{S}_{G}^{2}(0,T)$. 

In the following, we give estimates with respect to higher orders of $(Y^{\alpha},Z^{\alpha},K^{\alpha})$.

\begin{lemma}\label{ZandK}
	For fixed $\alpha,\lambda>0$, assume that $\xi\in L_{G}^{2+\lambda}(\Omega_{T})$ and $f$ satisfies (H1)-(H4). Let $f^{\alpha}$ be defined as in Lemma \ref{falpha} and let $(Y^{\alpha},Z^{\alpha},K^{\alpha})$ be the unique solution of $G$-BSDE (\ref{GBSDEalpha}). Then for $0\leq\bar{\lambda}\leq \lambda $,
	\begin{align*}
		&\left\| Z^{\alpha}\right\|_{H_{G}^{2+\bar{\lambda}}}^{2+\bar{\lambda}}\leq C_{T,\bar{\lambda},L,M,\bar{\sigma},\underline{\sigma}}\left\lbrace \left\| Y^{\alpha}\right\|_{S_{G}^{2+\bar{\lambda}}}^{2+\bar{\lambda}}+ \left\| Y^{\alpha}\right\|_{S_{G}^{2+\bar{\lambda}}}^{\frac{2+\bar{\lambda}}{2}}\left\|\int_{0}^{T}h(s)ds \right\|_{L_{G}^{2+\bar{\lambda}}}^{\frac{2+\bar{\lambda}}{2}}  \right\rbrace ,\\
		&\left\|K_{T} \right\|_{L_{G}^{2+\bar{\lambda}}}^{2+\bar{\lambda}}\leq C_{T,\bar{\lambda},L,M,\bar{\sigma},\underline{\sigma}}\left\lbrace \left\| Y^{\alpha}\right\|_{S_{G}^{2+\bar{\lambda}}}^{2+\bar{\lambda}}+ \left\|\int_{0}^{T}h(s)ds \right\|_{L_{G}^{2+\bar{\lambda}}}^{2+\bar{\lambda}}  \right\rbrace .
	\end{align*}
\end{lemma}

\begin{proof}
	Using It\^{o}'s formula to $\left\vert Y_{t}^{\alpha}\right\vert^{2}$, it follows that
	\begin{equation*}
		\left\vert Y_{0}^{\alpha }\right\vert ^{2}+\int_{0}^{T}\left\vert
		Z_{s}^{\alpha }\right\vert ^{2}d\langle B\rangle_{s}=\xi^{2}+\int_{0}^{T}2Y_{s}^{\alpha }f^{\alpha }\left( s,Y_{s}^{\alpha},Z_{s}^{\alpha }\right) ds-\int_{0}^{T}2Y_{s}^{\alpha}Z_{s}^{\alpha}dB_{s}-\int_{0}^{T}2Y_{s}^{\alpha }dK_{s}^{\alpha }.
	\end{equation*}%
	Through Cr's inequality, we get
	\begin{align*}
		&\left( \int_{0}^{T}\left\vert Z_{s}^{\alpha }\right\vert ^{2}d\langle B\rangle _{s}\right) ^{\frac{2+\bar{\lambda} }{2}} \\
		\leq &C_{\bar{\lambda}}\left\{ \xi ^{2+\bar{\lambda} }+\left\vert \int_{0}^{T}2Y_{s}^{\alpha }f^{\alpha}\left( s,Y_{s}^{\alpha },Z_{s}^{\alpha }\right) ds\right\vert ^{\frac{2+\bar{\lambda} }{2}}+\left\vert \int_{0}^{T}2Y_{s}^{\alpha }Z_{s}^{\alpha}dB_{s}\right\vert ^{\frac{2+\bar{\lambda} }{2}}+\left\vert\int_{0}^{T}2Y_{s}^{\alpha }dK_{s}^{\alpha }\right\vert ^{\frac{2+\bar{\lambda} }{2}}\right\}  \\
		\leq &C_{\bar{\lambda} }\left\{ \xi ^{2+\bar{\lambda} }+\left\vert\int_{0}^{T}2Y_{s}^{\alpha }f^{\alpha }\left( s,Y_{s}^{\alpha},Z_{s}^{\alpha }\right) ds\right\vert ^{2+\bar{\lambda} }+\left\vert\int_{0}^{T}2Y_{s}^{\alpha }Z_{s}^{\alpha }dB_{s}\right\vert ^{2+\bar{\lambda}}+\sup_{t\in \lbrack 0,T]}\left\vert Y_{t}^{\alpha }\right\vert ^{\frac{2+\bar{\lambda} }{2}}\left\vert K_{T}^{\alpha }\right\vert ^{\frac{2+\bar{\lambda} }{2}}\right\} .
	\end{align*}
	By (i)-(ii) and (v) in Lemma \ref{falpha}, we obtain 
	\begin{align*}
		2Y_{s}^{\alpha }f^{\alpha }\left( s,Y_{s}^{\alpha },Z_{s}^{\alpha }\right) =&2Y_{s}^{\alpha }\left( f^{\alpha }\left( s,Y_{s}^{\alpha },Z_{s}^{\alpha}\right) -f^{\alpha }\left( s,0,Z_{s}^{\alpha }\right) +f^{\alpha }\left(s,0,Z_{s}^{\alpha }\right) -f^{\alpha }\left( s,0,0\right) +f^{\alpha}\left( s,0,0\right) \right)  \\
		\leq &2u_{s}\left\vert Y_{s}^{\alpha }\right\vert ^{2}+2L\left\vert Y_{s}^{\alpha }Z_{s}^{\alpha }\right\vert +2\left\vert Y_{s}^{\alpha}f^{\alpha }\left( s,0,0\right) \right\vert  \\
		\leq &2u_{s}\left\vert Y_{s}^{\alpha }\right\vert ^{2}+2L\left\vert Y_{s}^{\alpha }Z_{s}^{\alpha }\right\vert +2\left\vert Y_{s}^{\alpha}f\left( s,0,0\right) \right\vert .
	\end{align*}
	With H\"{o}lder's inequality and (H3), it follows that 
	\begin{align*}
		&\left\vert \int_{0}^{T}2Y_{s}^{\alpha }f^{\alpha }\left( s,Y_{s}^{\alpha},Z_{s}^{\alpha }\right) ds\right\vert ^{\frac{2+\bar{\lambda} }{2}} \\
		\leq&\left\vert \int_{0}^{T}2u_{s}\left\vert Y_{s}^{\alpha }\right\vert^{2}+2L\left\vert Y_{s}^{\alpha }Z_{s}^{\alpha }\right\vert +2\left\vert Y_{s}^{\alpha }f\left( s,0,0\right) \right\vert ds\right\vert ^{\frac{2+\bar{\lambda} }{2}} \\
		\leq &C_{\bar{\lambda},L }\left\{ \left\vert \int_{0}^{T}u_{s}\left\vert
		Y_{s}^{\alpha }\right\vert ^{2}ds\right\vert ^{\frac{2+\bar{\lambda} }{2}}+\left\vert \int_{0}^{T}\left\vert Y_{s}^{\alpha }Z_{s}^{\alpha}\right\vert ds\right\vert ^{\frac{2+\bar{\lambda} }{2}}+\left\vert\int_{0}^{T}\left\vert Y_{s}^{\alpha }\right\vert h\left( s\right)ds\right\vert ^{\frac{2+\bar{\lambda} }{2}}\right\}  \\
		\leq &C_{\bar{\lambda} ,L,M}\left\{ \sup_{t\in \lbrack 0,T]}\left\vert
		Y_{t}^{\alpha }\right\vert ^{2+\bar{\lambda} }+\sup_{t\in \lbrack 0,T]}\left\vert Y_{t}^{\alpha }\right\vert ^{\frac{2+\bar{\lambda} }{2}}\left\vert\int_{0}^{T}\left\vert Z_{s}^{\alpha }\right\vert ds\right\vert ^{\frac{2+\bar{\lambda} }{2}}+\sup_{t\in \lbrack 0,T]}\left\vert Y_{t}^{\alpha}\right\vert ^{\frac{2+\bar{\lambda} }{2}}\left\vert \int_{0}^{T}h\left( s\right)
		ds\right\vert ^{\frac{2+\bar{\lambda} }{2}}\right\}  \\
		\leq &C_{\bar{\lambda} ,L,M}\left\{ \left( 1+\frac{1}{\epsilon }\right) \sup_{t\in\lbrack 0,T]}\left\vert Y_{t}^{\alpha }\right\vert ^{2+\bar{\lambda} }+\epsilon\left\vert \int_{0}^{T}\left\vert Z_{s}^{\alpha }\right\vert ds\right\vert ^{2+\bar{\lambda} }+\sup_{t\in \lbrack 0,T]}\left\vert Y_{t}^{\alpha}\right\vert ^{\frac{2+\bar{\lambda} }{2}}\left\vert \int_{0}^{T}h\left( s\right)ds\right\vert ^{\frac{2+\bar{\lambda} }{2}}\right\}  \\
		\leq &C_{T,\bar{\lambda} ,L,M}\left\{ \left( 1+\frac{1}{\epsilon }\right) \sup_{t\in\lbrack 0,T]}\left\vert Y_{t}^{\alpha }\right\vert ^{2+\bar{\lambda} }+\epsilon\left\vert \int_{0}^{T}\left\vert Z_{s}^{\alpha }\right\vert^{2}ds\right\vert ^{\frac{2+\bar{\lambda} }{2}}+\sup_{t\in \lbrack 0,T]}\left\vert Y_{t}^{\alpha }\right\vert ^{\frac{2+\bar{\lambda} }{2}}\left\vert\int_{0}^{T}h\left( s\right) ds\right\vert ^{\frac{2+\bar{\lambda} }{2}}\right\} .
	\end{align*}%
	Applying BDG's inequality, we yield 
	\begin{align*}
		\hat{\mathbb{E}}\left[ \left\vert \int_{0}^{T}2Y_{s}^{\alpha }Z_{s}^{\alpha}dB_{s}\right\vert ^{\frac{2+\bar{\lambda} }{2}}\right]  \leq &C_{\bar{\lambda} }\hat{\mathbb{E}}\left[ \left\vert \int_{0}^{T}\left\vert Y_{s}^{\alpha}Z_{s}^{\alpha }\right\vert ^{2}d\langle B\rangle _{s}\right\vert^{\frac{2+\bar{\lambda} }{4}}\right]  \\
		\leq &C_{\bar{\lambda} ,\bar{\sigma}}\hat{\mathbb{E}}\left[ \sup_{t\in \lbrack0,T]}\left\vert Y_{t}^{\alpha }\right\vert ^{\frac{2+\bar{\lambda} }{2}}\left\vert\int_{0}^{T}\left\vert Z_{s}^{\alpha }\right\vert ^{2}ds\right\vert ^{\frac{2+\bar{\lambda} }{4}}\right]  \\
		\leq &C_{\bar{\lambda} ,\bar{\sigma}}\hat{\mathbb{E}}\left[ \frac{1}{\epsilon }\sup_{t\in \lbrack 0,T]}\left\vert Y_{t}^{\alpha }\right\vert ^{2+\bar{\lambda} }\right] +C_{\bar{\lambda} ,\bar{\sigma}}\hat{\mathbb{E}}\left[ \epsilon \left\vert\int_{0}^{T}\left\vert Z_{s}^{\alpha }\right\vert ^{2}ds\right\vert ^{\frac{2+\bar{\lambda} }{2}}\right] .
	\end{align*}
	So by simple calculation, 
	\begin{align}
		\hat{\mathbb{E}}\left[ \left( \int_{0}^{T}\left\vert Z_{s}^{\alpha}\right\vert ^{2}d\langle B\rangle _{s}\right) ^{\frac{2+\bar{\lambda} }{2}}\right] 
		\leq &C_{T,\bar{\lambda} ,L,M,\bar{\sigma}}\left\{ \left( 1+\frac{1}{\epsilon }\right) \left\Vert Y^{\alpha }\right\Vert _{S_{G}^{2+\bar{\lambda} }}^{2+\bar{\lambda}}+\epsilon \left\Vert Z^{\alpha }\right\Vert _{H_{G}^{2+\bar{\lambda}}}^{2+\bar{\lambda} }\right\}\notag \\
		&+C_{T,\bar{\lambda} ,L,M,\bar{\sigma}}\left\{\left\Vert Y^{\alpha }\right\Vert _{S_{G}^{2+\bar{\lambda} }}^{\frac{2+\bar{\lambda} }{2}}\left( \left\Vert \int_{0}^{T}h(s)ds\right\Vert_{L_{G}^{2+\bar{\lambda} }}^{\frac{2+\bar{\lambda} }{2}}+\left\Vert K_{T}^{\alpha}\right\Vert _{L_{G}^{2+\bar{\lambda} }}^{\frac{2+\bar{\lambda} }{2}}\right) \right\} .
		\label{Z}
	\end{align}
	
	Now let us estimate $K_{T}^{\alpha }$. Note that 
	\begin{equation*}
		K_{T}^{\alpha }=\xi -Y_{0}^{\alpha }-\int_{0}^{T}f^{\alpha }(s,Y_{s}^{\alpha},Z_{s}^{\alpha })ds-\int_{0}^{T}Z_{s}^{\alpha }dBs.
	\end{equation*}%
	Due to (ii) and (v) in Lemma \ref{falpha}, it leads to
	\begin{align*}
		\left\vert f^{\alpha }(s,Y_{s}^{\alpha },Z_{s}^{\alpha })\right\vert =&\left\vert f^{\alpha }(s,Y_{s}^{\alpha },Z_{s}^{\alpha })-f^{\alpha}(s,Y_{s}^{\alpha },0)+f^{\alpha }(s,Y_{s}^{\alpha },0)\right\vert  \\
		\leq &h(s)+3u_{s}\left\vert Y_{s}^{\alpha }\right\vert +L\left\vert Z_{s}^{\alpha }\right\vert .
	\end{align*}%
	By a similar approach to the above, we can obtain 
	\begin{align}
		\hat{\mathbb{E}}\left[ \left\vert K_{T}^{\alpha }\right\vert ^{2+\bar{\lambda} }\right]  \leq &C_{\bar{\lambda} }\left\{ \hat{\mathbb{E}}\left[ \left\vert \xi\right\vert ^{2+\bar{\lambda} }\right] +\hat{\mathbb{E}}\left[ \left\vert Y_{0}^{\alpha }\right\vert ^{2+\bar{\lambda} }\right] +\hat{\mathbb{E}}\left[\left\vert \int_{0}^{T}f^{\alpha }(s,Y_{s}^{\alpha },Z_{s}^{\alpha})ds\right\vert ^{2+\bar{\lambda} }\right] +\hat{\mathbb{E}}\left[ \left\vert\int_{0}^{T}Z_{s}^{\alpha }dBs\right\vert ^{2+\bar{\lambda} }\right] \right\}  \notag \\
		\leq &C_{\bar{\lambda} }\left\{ \left\Vert Y^{\alpha }\right\Vert
		_{S_{G}^{2+\bar{\lambda} }}^{2+\bar{\lambda} }+\hat{\mathbb{E}}\left[ \left\vert\int_{0}^{T}h(s)+u_{s}\left\vert Y_{s}^{\alpha }\right\vert+L\left\vert Z_{s}^{\alpha }\right\vert ds\right\vert ^{2+\bar{\lambda} }\right] +\hat{\mathbb{E}}\left[ \left\vert \int_{0}^{T}\left\vert Z_{s}^{\alpha }\right\vert ^{2}d\langle B\rangle s\right\vert ^{\frac{2+\bar{\lambda} }{2}}\right] \right\}   \notag \\
		\leq &C_{T,\bar{\lambda} ,L,M,\bar{\sigma}}\left\{ \left\Vert Y^{\alpha }\right\Vert_{S_{G}^{2+\bar{\lambda} }}^{2+\bar{\lambda} }+\left\Vert \int_{0}^{T}h(s)ds\right\Vert_{L_{G}^{2+\bar{\lambda} }}^{2+\bar{\lambda} }+\left\Vert Z^{\alpha }\right\Vert_{H_{G}^{2+\bar{\lambda} }}^{2+\bar{\lambda} }\right\} .  \label{K}
	\end{align}%
	Combining (\ref{Z}), it follows that 
	\begin{align*}
		&\hat{\mathbb{E}}\left[ \left( \int_{0}^{T}\left\vert Z_{s}^{\alpha}\right\vert ^{2}d\langle B\rangle _{s}\right) ^{\frac{2+\bar{\lambda} }{2}}\right]  \\
		\leq &C_{T,\bar{\lambda} ,L,M,\bar{\sigma}}\left\{ \left( 1+\frac{1}{\epsilon }\right) \left\Vert Y^{\alpha }\right\Vert _{S_{G}^{2+\bar{\lambda}
		}}^{2+\bar{\lambda} }+\epsilon \left\Vert Z^{\alpha }\right\Vert
		_{H_{G}^{2+\bar{\lambda} }}^{2+\bar{\lambda} }+\left\Vert Y^{\alpha }\right\Vert_{S_{G}^{2+\bar{\lambda} }}^{\frac{2+\bar{\lambda} }{2}}\left( \left\Vert\int_{0}^{T}h(s)ds\right\Vert _{L_{G}^{2+\bar{\lambda} }}^{\frac{2+\bar{\lambda} }{2}}+\left\Vert K_{T}^{\alpha }\right\Vert _{L_{G}^{2+\bar{\lambda} }}^{\frac{2+\bar{\lambda} }{2}}\right) \right\}  \\
		\leq &C_{T,\bar{\lambda} ,L,M,\bar{\sigma}}\left\{ \left( 1+\frac{1}{\epsilon }\right) \left\Vert Y^{\alpha }\right\Vert _{S_{G}^{2+\bar{\lambda} }}^{2+\bar{\lambda}}+\epsilon \left\Vert Z^{\alpha }\right\Vert _{H_{G}^{2+\bar{\lambda}}}^{2+\bar{\lambda} }+\left\Vert Y^{\alpha }\right\Vert _{S_{G}^{2+\bar{\lambda} }}^{\frac{2+\bar{\lambda} }{2}}\left( \left\Vert \int_{0}^{T}h(s)ds\right\Vert_{L_{G}^{2+\bar{\lambda} }}^{\frac{2+\bar{\lambda} }{2}}+\left\Vert Z^{\alpha}\right\Vert _{H_{G}^{2+\bar{\lambda} }}^{\frac{2+\bar{\lambda} }{2}}\right) \right\}  \\
		\leq &C_{T,\bar{\lambda} ,L,M,\bar{\sigma}}\left\{ \left( 1+\frac{1}{\epsilon }\right) \left\Vert Y^{\alpha }\right\Vert _{S_{G}^{2+\bar{\lambda} }}^{2+\bar{\lambda}}+\epsilon \left\Vert Z^{\alpha }\right\Vert _{H_{G}^{2+\bar{\lambda}}}^{2+\bar{\lambda} }+\left\Vert Y^{\alpha }\right\Vert _{S_{G}^{2+\bar{\lambda} }}^{\frac{2+\bar{\lambda} }{2}}\left\Vert \int_{0}^{T}h(s)ds\right\Vert_{L_{G}^{2+\bar{\lambda} }}^{\frac{2+\bar{\lambda} }{2}}\right\} .
	\end{align*}
	Let $\epsilon $ satisfy that $C_{T,\bar{\lambda},L,M,\bar{\sigma}}\epsilon =\frac{\underline{\sigma }^{2}}{2}$. Then 
	\begin{equation*}
		\frac{\underline{\sigma }^{2}}{2}  \hat{\mathbb{E}}\left[ \left( \int_{0}^{T}\left\vert Z_{s}^{\alpha}\right\vert ^{2}ds\right) ^{\frac{2+\bar{\lambda} }{2}}\right] \leq C_{T,\bar{\lambda},L,M,\bar{\sigma}}\left\{ \left( 1+\frac{1}{\epsilon }\right) \left\Vert Y^{\alpha }\right\Vert _{S_{G}^{2+\bar{\lambda} }}^{2+\bar{\lambda} }+\left\Vert Y^{\alpha }\right\Vert _{S_{G}^{2+\bar{\lambda} }}^{\frac{2+\bar{\lambda} }{2}}\left\Vert \int_{0}^{T}h(s)ds\right\Vert _{L_{G}^{2+\bar{\lambda} }}^{\frac{2+\bar{\lambda} }{2}}\right\} .
	\end{equation*}%
	It follows that 
	\begin{equation}
		\hat{\mathbb{E}}\left[ \left( \int_{0}^{T}\left\vert Z_{s}^{\alpha}\right\vert ^{2}ds\right) ^{\frac{2+\bar{\lambda} }{2}}\right] \leq C_{T,\bar{\lambda},L ,M,\bar{\sigma},\underline{\sigma }^{2}}\left\{ \left\Vert Y^{\alpha}\right\Vert _{S_{G}^{2+\bar{\lambda} }}^{2+\bar{\lambda} }+\left\Vert Y^{\alpha}\right\Vert _{S_{G}^{2+\bar{\lambda} }}^{\frac{2+\bar{\lambda} }{2}}\left\Vert\int_{0}^{T}h(s)ds\right\Vert _{L_{G}^{2+\bar{\lambda} }}^{\frac{2+\bar{\lambda} }{2}}\right\} .  \label{Z no K}
	\end{equation}
	Finally, combining (\ref{K}) with (\ref{Z no K}), we get the desired result.
\end{proof}

\begin{lemma}\label{Yalpha}
	For fixed $\alpha ,\lambda >0$, assume that $\xi \in L_{G}^{2+\lambda }$ and $f$ satisfies (H1)-(H4). Let $f^{\alpha }$ be defined as in Lemma \ref{falpha} and let $(Y^{\alpha },Z^{\alpha },K^{\alpha })$ be the unique solution of $G$-BSDE (\ref{GBSDEalpha}). Then
	\begin{equation*}
		\left\vert Y_{t}^{\alpha }\right\vert \leq C_{T,L,M,\underline{\sigma }}\left\{ \left( \hat{\mathbb{E}}_{t}\left[ \xi ^{2}\right] \right) ^{\frac{1}{2}}+\left( \hat{\mathbb{E}}_{t}\left[ \left( \int_{t}^{T}\left\vert h(s)\right\vert ^{2}ds\right) \right] \right) ^{\frac{1}{2}}\right\} .
	\end{equation*}
\end{lemma}

\begin{proof}
	Set $\theta _{s}:=u_{s}+\frac{1}{\underline{\sigma}^{2}}L^{2}+1$.
	Clearly, $\int_{0}^{T}\theta _{s}ds<\infty $. Applying It\^{o}'s formula to $\left\vert Y_{t}^{\alpha }e^{\int_{0}^{t}\theta _{s}ds}\right\vert ^{2}$, we get
	\begin{align}
		&\left\vert Y_{t}^{\alpha }\right\vert ^{2}e^{2\int_{0}^{t}\theta_{r}dr}+\int_{t}^{T}2\left\vert Y_{s}^{\alpha }\right\vert ^{2}\theta_{s}e^{2\int_{0}^{s}\theta _{r}dr}ds+\int_{t}^{T}\left\vert Z_{s}^{\alpha}\right\vert ^{2}e^{2\int_{0}^{s}\theta _{r}dr}d\langle B\rangle_{s}  \notag \\
		=&\xi^{2}e^{2\int_{0}^{T}\theta_{r}dr}+\int_{t}^{T}2Y_{s}^{\alpha}f^{\alpha }\left( s,Y_{s}^{\alpha },Z_{s}^{\alpha }\right)e^{2\int_{0}^{s}\theta_{r}dr}ds-\int_{t}^{T}2Y_{s}^{\alpha}Z_{s}^{\alpha}e^{2\int_{0}^{s}\theta_{r}dr}dB_{s}-\int_{t}^{T}2Y_{s}^{\alpha}e^{2\int_{0}^{s}\theta _{r}dr}dK_{s}^{\alpha }  \notag \\
		\leq&\xi^{2}e^{2\int_{0}^{T}\theta_{r}dr}+\int_{t}^{T}2Y_{s}^{\alpha}f^{\alpha }\left( s,Y_{s}^{\alpha },Z_{s}^{\alpha }\right)
		e^{2\int_{0}^{s}\theta _{r}dr}ds-\left( N_{T}-N_{t}\right) ,  \label{Y e}
	\end{align}%
	where $N_{t}=\int_{0}^{t}2Y_{s}^{\alpha}Z_{s}^{\alpha}e^{2\int_{0}^{s}\theta_{r}dr}dB_{s}+\int_{0}^{t}2Y_{s}^{\alpha,+}e^{2\int_{0}^{s}\theta _{r}dr}dK_{s}^{\alpha }$. From Lemma 3.4 in \cite{hu2014backward}, we konw that $N$ is a martingale. According to Lemma \ref{falpha} and $2ab\leq a^{2}+b^{2}$, it follows that 
	\begin{align*}
		&\int_{t}^{T}2Y_{s}^{\alpha }f^{\alpha }\left( s,Y_{s}^{\alpha
		},Z_{s}^{\alpha }\right) e^{2\int_{0}^{s}\theta _{r}dr}ds \\
		\leq &\int_{t}^{T}\left( 2u_{s}\left\vert Y_{s}^{\alpha }\right\vert^{2}+2L\left\vert Y_{s}^{\alpha }Z_{s}^{\alpha }\right\vert +2\left\vert Y_{s}^{\alpha }h(s)\right\vert \right) e^{2\int_{0}^{s}\theta _{r}dr}ds \\
		\leq &\int_{t}^{T}\left( 2u_{s}+\frac{1}{\underline{\sigma }^{2}}L^{2}+1\right) \left\vert Y_{s}^{\alpha }\right\vert
		^{2}e^{2\int_{0}^{s}\theta _{r}dr}ds+\int_{t}^{T}\left( \underline{\sigma }^{2}\left\vert Z_{s}^{\alpha }\right\vert ^{2}+\left\vert h(s)\right\vert^{2}\right)e^{2\int_{0}^{s}\theta _{r}dr}ds.
	\end{align*}%
	Then (\ref{Y e}) turns to 
	\begin{align*}
		&\left\vert Y_{t}^{\alpha }\right\vert^{2}e^{2\int_{0}^{t}\theta_{r}dr}+\int_{t}^{T}\left\vert Y_{s}^{\alpha }\right\vert ^{2}\left( 2\theta_{s}-2u_{s}-\frac{1}{\underline{\sigma}^{2}}L^{2}-1\right)e^{2\int_{0}^{s}\theta _{r}dr}ds+\int_{t}^{T}\left\vert Z_{s}^{\alpha}\right\vert ^{2}e^{2\int_{0}^{s}\theta _{r}dr}\left( d\langle B\rangle _{s}-\underline{\sigma }^{2}ds\right)  \\
		\leq&\xi^{2}e^{2\int_{0}^{T}\theta_{r}dr}+\int_{t}^{T}\left\vert
		h(s)\right\vert^{2}e^{2\int_{0}^{s}\theta_{r}dr}ds-\left(N_{T}-N_{t}\right) .
	\end{align*}%
	Thus, 
	\begin{equation*}
		\left\vert Y_{t}^{\alpha}\right\vert^{2}e^{2\int_{0}^{t}\theta_{r}dr}+N_{T}-N_{t}\leq\xi^{2}e^{2\int_{0}^{T}\theta_{r}dr}+\int_{t}^{T}\left\vert h(s)\right\vert ^{2}e^{2\int_{0}^{s}\theta_{r}dr}ds.
	\end{equation*}%
	Taking the conditional expectation on both sides of the inequality, it leads to 
	\begin{equation*}
		\left\vert Y_{t}^{\alpha }\right\vert ^{2}e^{2\int_{0}^{t}\theta _{r}dr}\leq
		e^{2\int_{0}^{T}\theta _{r}dr}\left( \hat{\mathbb{E}}_{t}\left[ \xi ^{2}%
		\right] +\hat{\mathbb{E}}_{t}\left[ \int_{t}^{T}\left\vert h(s)\right\vert
		^{2}ds\right] \right) .
	\end{equation*}%
	Recall that $1\leq e^{2\int_{0}^{T}\theta _{r}dr}\leq C_{T,L,M,\underline{\sigma }}<\infty $. Therefore, 
	\begin{equation*}
		\left\vert Y_{t}^{\alpha }\right\vert ^{2}\leq C_{T,L,M,\underline{\sigma }}\left( \hat{\mathbb{E}}_{t}\left[ \xi ^{2}\right] +\hat{\mathbb{E}}_{t}\left[ \int_{t}^{T}\left\vert h(s)\right\vert ^{2}ds\right] \right) .
	\end{equation*}%
	The proof is completed.
\end{proof}

\begin{corollary}
	Due to the above Lemma, it is simple to show that 
	\begin{equation*}
		\left\vert Y_{t}^{\alpha }\right\vert ^{2+\bar{\lambda} }\leq C_{T,\bar{\lambda},L,M,\underline{\sigma}}\left\{\hat{\mathbb{E}}_{t}\left[ \xi ^{2+\bar{\lambda} }\right] +\hat{\mathbb{E}}_{t}\left[ \left( \int_{t}^{T}\left\vert h(s)\right\vert ^{2+\bar{\lambda} }ds\right) \right] \right\},
	\end{equation*}
	where $0\leq\bar{\lambda}\leq\lambda$.
\end{corollary}

It is worth noting that through Lemmas \ref{ZandK} and \ref{Yalpha}, we
obtain that the estimates of $(Y^{\alpha },Z^{\alpha },K^{\alpha })$ are
independent of $\alpha $. The following two lemmas play a crucial role in our approach.

\begin{lemma}\label{Zcha}
	For fixed $\alpha ,\beta ,\lambda >0$, assume that $\xi \in L_{G}^{2+\lambda
	}$ and $f$ satisfies (H1)-(H4). Let $f^{\alpha },f^{\beta }$ be defined as
	in Lemma \ref{falpha}. Let the corresponding solutions of the generators $%
	f^{\alpha }$ and $f^{\beta }$ of $G$-BSDE (\ref{GBSDEalpha}) be $(Y^{\alpha
	},Z^{\alpha },K^{\alpha })$ and $(Y^{\beta },Z^{\beta },K^{\beta })$,
	respectively. Define $\hat{Y}^{\alpha ,\beta }=Y^{\alpha }-Y^{\beta },\hat{Z}%
	^{\alpha ,\beta }=Z^{\alpha }-Z^{\beta }$. Then 
	\begin{equation*}
		\Vert \hat{Z}^{\alpha ,\beta }\Vert _{H_{G}^{2}}^{2}\leq C_{T,L,M,\bar{\sigma},\underline{\sigma }}\left\{ \Vert \hat{Y}^{\alpha ,\beta}\Vert _{S_{G}^{2}}^{2}+\Vert \hat{Y}^{\alpha ,\beta }\Vert
		_{S_{G}^{2}}\left(\left\Vert\int_{0}^{T}h(s)ds\right\Vert_{L_{G}^{2}}+\left\Vert Y^{\alpha }\right\Vert _{S_{G}^{2}}+\left\Vert Y^{\beta }\right\Vert _{S_{G}^{2}}\right) \right\} .
	\end{equation*}
\end{lemma}

\begin{proof}
	We shall adopt the procedure as in the proof of Lemma \ref{ZandK}. Set $\hat{%
		K}^{\alpha ,\beta }=K^{\alpha }-K^{\beta }$. Applying It\^{o}'s formula to $%
	\vert \hat{Y}^{\alpha ,\beta }\vert ^{2}$, we get 
	\begin{align*}
		&\left\vert \hat{Y}_{T}^{\alpha ,\beta}\right\vert^{2}-\left\vert \hat{Y}_{0}^{\alpha,\beta}\right\vert^{2}+\int_{0}^{T}2\hat{Y}_{s}^{\alpha,\beta }\left( f^{\alpha }(s,Y_{s}^{\alpha },Z_{s}^{\alpha })-f^{\beta}(s,Y_{s}^{\beta },Z_{s}^{\beta })\right) ds\\
		=&\int_{0}^{T}2\hat{Y}_{s}^{\alpha ,\beta }\hat{Z}_{s}^{\alpha ,\beta }dB_{s}+\int_{0}^{T}2\hat{Y}_{s}^{\alpha ,\beta }d\hat{K}_{s}^{\alpha ,\beta }+\int_{0}^{T}\left\vert 
		\hat{Z}_{s}^{\alpha ,\beta }\right\vert ^{2}d\langle B\rangle s.
	\end{align*}%
	It turns to 
	\begin{equation*}
		\underline{\sigma}^{2}\int_{0}^{T}\left\vert\hat{Z}_{s}^{\alpha,\beta}\right\vert ^{2}ds\leq \int_{0}^{T}2\hat{Y}_{s}^{\alpha ,\beta }\left(f^{\alpha }(s,Y_{s}^{\alpha },Z_{s}^{\alpha })-f^{\beta }(s,Y_{s}^{\beta},Z_{s}^{\beta })\right) ds-\int_{0}^{T}2\hat{Y}_{s}^{\alpha ,\beta }\hat{Z}_{s}^{\alpha ,\beta}dB_{s}-\int_{0}^{T}2\hat{Y}_{s}^{\alpha,\beta}d\hat{K}_{s}^{\alpha ,\beta }.
	\end{equation*}%
	According to Lemma \ref{falpha}, we yield%
	\begin{align*}
		&\hat{Y}_{s}^{\alpha ,\beta }\left( f^{\alpha }(s,Y_{s}^{\alpha
		},Z_{s}^{\alpha })-f^{\beta }(s,Y_{s}^{\beta },Z_{s}^{\beta })\right)  \\
		\leq &\hat{Y}_{s}^{\alpha,\beta}\left(f^{\alpha}(s,Y_{s}^{\alpha
		},Z_{s}^{\alpha })-f^{\alpha }(s,Y_{s}^{\beta },Z_{s}^{\alpha })+f^{\alpha}(s,Y_{s}^{\beta },Z_{s}^{\alpha })-f^{\alpha }(s,Y_{s}^{\beta},Z_{s}^{\beta })+f^{\alpha }(s,Y_{s}^{\beta },Z_{s}^{\beta })-f^{\beta}(s,Y_{s}^{\beta },Z_{s}^{\beta })\right)  \\
		\leq &u_{s}\left\vert \hat{Y}_{s}^{\alpha ,\beta }\right\vert^{2}+L\left\vert \hat{Y}_{s}^{\alpha ,\beta }\hat{Z}_{s}^{\alpha ,\beta}\right\vert +\hat{Y}_{s}^{\alpha ,\beta }\left( f^{\alpha }(s,Y_{s}^{\beta},Z_{s}^{\beta })-f^{\beta }(s,Y_{s}^{\beta },Z_{s}^{\beta })\right)  \\
		\leq &u_{s}\left\vert \hat{Y}_{s}^{\alpha ,\beta }\right\vert^{2}+L\left\vert \hat{Y}_{s}^{\alpha ,\beta }\hat{Z}_{s}^{\alpha ,\beta}\right\vert +\hat{Y}_{s}^{\alpha ,\beta }\left( 6u_{s}\left\vert Y_{s}^{\beta }\right\vert +2L\left\vert Z_{s}^{\beta }\right\vert+2h(s)\right) .
	\end{align*}%
	Then, through simple calculation, it follows that%
	\begin{align*}
		&\hat{\mathbb{E}}\left[ \int_{0}^{T}2\hat{Y}_{s}^{\alpha ,\beta }\left(f^{\alpha }(s,Y_{s}^{\alpha },Z_{s}^{\alpha })-f^{\beta }(s,Y_{s}^{\beta},Z_{s}^{\beta })\right) ds\right]  \\
		\leq &\hat{\mathbb{E}}\left[ \int_{0}^{T}u_{s}\left\vert \hat{Y}_{s}^{\alpha ,\beta }\right\vert ^{2}+L\left\vert \hat{Y}_{s}^{\alpha,\beta }\hat{Z}_{s}^{\alpha ,\beta }\right\vert +\hat{Y}_{s}^{\alpha ,\beta}\left( 6u_{s}\left\vert Y_{s}^{\beta }\right\vert +2L\left\vert Z_{s}^{\beta }\right\vert +2h(s)\right) ds\right]  \\
		\leq &C_{L,M}\left\{ \left( 1+\frac{1}{\epsilon }\right) \Vert \hat{Y}^{\alpha ,\beta }\Vert _{S_{G}^{2}}^{2}+\epsilon \Vert \hat{Z}^{\alpha ,\beta }\Vert _{H_{G}^{2}}^{2}+\Vert \hat{Y}^{\alpha,\beta }\Vert _{S_{G}^{2}}\left( \left\Vert Y^{\beta }\right\Vert_{S_{G}^{2}}+\left\Vert Z^{\beta}\right\Vert_{H_{G}^{2}}+\left\Vert\int_{0}^{T}h(s)ds\right\Vert _{L_{G}^{2}}\right) \right\} .
	\end{align*}%
	Clearly, 
	\begin{equation*}
		\hat{\mathbb{E}}\left[ \int_{0}^{T}2\hat{Y}_{s}^{\alpha ,\beta }d\hat{K}_{s}^{\alpha ,\beta }\right] \leq C\Vert \hat{Y}^{\alpha ,\beta}\Vert _{S_{G}^{2}}\left( \left\Vert K_{T}^{\alpha }\right\Vert_{L_{G}^{2}}+\Vert K_{T}^{\beta }\Vert _{L_{G}^{2}}\right) .
	\end{equation*}%
	By BDG's inequality, we have%
	\begin{align*}
		\hat{\mathbb{E}}\left[ \int_{0}^{T}2\hat{Y}_{s}^{\alpha ,\beta }\hat{Z}_{s}^{\alpha ,\beta }dB_{s}\right]  \leq &C_{\bar{\sigma}}\hat{\mathbb{E}}
		\left[ \left( \int_{0}^{T}\left\vert \hat{Y}_{s}^{\alpha ,\beta }\hat{Z}_{s}^{\alpha ,\beta }\right\vert ^{2}ds\right) ^{\frac{1}{2}}\right]  \\
		\leq &C_{\bar{\sigma}}\left( \frac{1}{\epsilon }\Vert \hat{Y}^{\alpha,\beta }\Vert _{S_{G}^{2}}^{2}+\epsilon \Vert \hat{Z}^{\alpha,\beta }\Vert _{H_{G}^{2}}^{2}\right) .
	\end{align*}%
	Thus, 
	\begin{align*}
		&\underline{\sigma }^{2}\hat{\mathbb{E}}\left[ \int_{0}^{T}\left\vert \hat{Z}_{s}^{\alpha ,\beta }\right\vert ^{2}ds\right]  \\
		\leq &C_{L,M,\bar{\sigma}}\left\{ \Vert \hat{Y}^{\alpha ,\beta }\Vert _{S_{G}^{2}}\left( \left\Vert Y^{\beta
		}\right\Vert _{S_{G}^{2}}+\left\Vert Z^{\beta }\right\Vert
		_{H_{G}^{2}}+\left\Vert K_{T}^{\alpha }\right\Vert_{L_{G}^{2}}+\Vert K_{T}^{\beta }\Vert _{L_{G}^{2}}+\left\Vert\int_{0}^{T}h(s)ds\right\Vert _{L_{G}^{2}}\right) \right\}\\
		&+C_{L,M,\bar{\sigma}}\left\lbrace\left( 1+\frac{1}{\epsilon }\right)\Vert \hat{Y}^{\alpha ,\beta }\Vert _{S_{G}^{2}}^{2}+\epsilon\Vert \hat{Z}^{\alpha ,\beta }\Vert _{H_{G}^{2}}^{2} \right\rbrace  .
	\end{align*}%
	Put $C_{L,M,\bar{\sigma}}\epsilon =\frac{\underline{\sigma }^{2}}{2}$.
	Combined with Lemma \ref{ZandK}, we obtain%
	\begin{equation*}
		\hat{\mathbb{E}}\left[ \int_{0}^{T}\left\vert \hat{Z}_{s}^{\alpha ,\beta}\right\vert ^{2}ds\right] \leq C_{T,L,M,\bar{\sigma},\underline{\sigma }}\left\{ \Vert \hat{Y}^{\alpha ,\beta }\Vert_{S_{G}^{2}}^{2}+\Vert \hat{Y}^{\alpha ,\beta }\Vert_{S_{G}^{2}}\left( \left\Vert \int_{0}^{T}h(s)ds\right\Vert_{L_{G}^{2}}+\left\Vert Y^{\alpha }\right\Vert _{S_{G}^{2}}+\left\Vert Y^{\beta }\right\Vert _{S_{G}^{2}}\right) \right\} .
	\end{equation*}
	This completes the proof. 
\end{proof}

\begin{lemma}\label{Ycha}
	For fixed $\alpha ,\beta ,\lambda >0$, assume that $\xi \in L_{G}^{2+\lambda}$ and $f$ satisfies (H1)-(H4). Let $f^{\alpha },f^{\beta },\hat{Y}^{\alpha,\beta },\hat{Z}_{s}^{\alpha ,\beta }$ be defined as in Lemma \ref{Zcha}.
	Then%
	\begin{equation*}
		\left\vert \hat{Y}_{t}^{\alpha ,\beta }\right\vert ^{2}\leq \left( \alpha+\beta \right) C_{T,L,M,\underline{\sigma }}\hat{\mathbb{E}}_{t}\left[\int_{t}^{T}\left\vert Z_{s}^{\alpha }\right\vert ^{2}+\left\vert h(s)\right\vert ^{2}+u_{s}^{2}\left\vert Y_{s}^{\alpha}\right\vert^{2}+u_{s}^{2}\left\vert Y_{s}^{\beta }\right\vert ^{2}ds\right].
	\end{equation*}
\end{lemma}

\begin{proof}
	Set $\theta _{s}=u_{s}+\frac{1}{\underline{\sigma }^{2}}L^{2}$. Using It\^{o}'s formula to $\left\vert \hat{Y}_{t}^{\alpha ,\beta }e^{\int_{0}^{t}\theta_{r}dr}\right\vert ^{2}$, we get
	\begin{align*}
		&\left\vert \hat{Y}_{t}^{\alpha ,\beta }\right\vert
		^{2}e^{2\int_{0}^{t}\theta _{r}dr}+\int_{t}^{T}2\theta _{s}\left\vert \hat{Y}_{s}^{\alpha ,\beta }\right\vert ^{2}e^{2\int_{0}^{s}\theta_{r}dr}ds+\int_{t}^{T}\left\vert \hat{Z}_{s}^{\alpha ,\beta }\right\vert^{2}e^{2\int_{0}^{s}\theta _{r}dr}d\langle B\rangle s \\
		=&\int_{t}^{T}2\hat{Y}_{s}^{\alpha ,\beta }\left( f^{\alpha
		}(s,Y_{s}^{\alpha },Z_{s}^{\alpha })-f^{\beta }(s,Y_{s}^{\beta
		},Z_{s}^{\beta })\right) e^{2\int_{0}^{s}\theta _{r}dr}ds-\int_{t}^{T}2\hat{Y}_{s}^{\alpha ,\beta }\hat{Z}_{s}^{\alpha,\beta}e^{2\int_{0}^{s}\theta_{r}dr}dB_{s}\\
		&-\int_{t}^{T}2\hat{Y}_{s}^{\alpha ,\beta}e^{2\int_{0}^{s}\theta _{r}dr}d\left( K_{s}^{\alpha }-K_{s}^{\beta }\right) 
		\\
		\leq &\int_{t}^{T}2\hat{Y}_{s}^{\alpha ,\beta }\left( f^{\alpha
		}(s,Y_{s}^{\alpha },Z_{s}^{\alpha })-f^{\beta }(s,Y_{s}^{\beta
		},Z_{s}^{\beta })\right) e^{2\int_{0}^{s}\theta _{r}dr}ds-\left(
		N_{T}-N_{t}\right) ,
	\end{align*}%
	where $N_{t}=\int_{0}^{t}2\hat{Y}_{s}^{\alpha,\beta}\hat{Z}_{s}^{\alpha,\beta}e^{2\int_{0}^{s}\theta_{r}dr}dB_{s}+\int_{0}^{t}2\hat{Y}_{s}^{\alpha,\beta,+}e^{2\int_{0}^{s}\theta_{r}dr}dK_{s}^{\alpha}+\int_{0}^{t}2\hat{Y}_{s}^{\alpha,\beta,-}e^{2\int_{0}^{s}\theta_{r}dr}dK_{s}^{\beta }$. From Lemma \ref{falpha}, we yield
	\begin{align*}
		&2\hat{Y}_{s}^{\alpha ,\beta }\left( f^{\alpha }(s,Y_{s}^{\alpha
		},Z_{s}^{\alpha })-f^{\beta }(s,Y_{s}^{\beta },Z_{s}^{\beta })\right)  \\
		\leq&2\hat{Y}_{s}^{\alpha,\beta}\left(f^{\alpha}(s,Y_{s}^{\alpha
		},Z_{s}^{\alpha })-f^{\beta }(s,Y_{s}^{\beta },Z_{s}^{\alpha })+f^{\beta}(s,Y_{s}^{\beta },Z_{s}^{\alpha })-f^{\beta }(s,Y_{s}^{\beta },Z_{s}^{\beta})\right)  \\
		\leq &2\left( \alpha +\beta \right) \left\{ \left\vert f^{\alpha
		}(s,Y_{s}^{\alpha },Z_{s}^{\alpha })\right\vert +\left\vert f^{\beta}(s,Y_{s}^{\beta },Z_{s}^{\alpha })\right\vert +u_{s}\left( \left\vert Y_{s}^{\alpha }\right\vert +\left\vert Y_{s}^{\beta }\right\vert \right)\right\} ^{2}+2u_{s}\left\vert \hat{Y}_{s}^{\alpha ,\beta }\right\vert^{2}+2L\left\vert \hat{Y}_{s}^{\alpha,\beta}\hat{Z}_{s}^{\alpha,\beta}\right\vert  \\
		\leq &\left( \alpha +\beta \right) C\left( \left\vert f^{\alpha
		}(s,Y_{s}^{\alpha },Z_{s}^{\alpha })\right\vert ^{2}+\left\vert f^{\beta}(s,Y_{s}^{\beta },Z_{s}^{\alpha })\right\vert ^{2}+u_{s}^{2}\left\vert Y_{s}^{\alpha }\right\vert ^{2}+u_{s}^{2}\left\vert Y_{s}^{\beta}\right\vert ^{2}\right) +2u_{s}\left\vert \hat{Y}_{s}^{\alpha ,\beta }\right\vert^{2}+2L\left\vert\hat{Y}_{s}^{\alpha,\beta}\hat{Z}_{s}^{\alpha,\beta}\right\vert\\
		\leq &\left( \alpha +\beta \right) C_{L}\left( \left\vert Z_{s}^{\alpha}\right\vert ^{2}+\left\vert h(s)\right\vert ^{2}+u_{s}^{2}\left\vert Y_{s}^{\alpha }\right\vert ^{2}+u_{s}^{2}\left\vert Y_{s}^{\beta}\right\vert ^{2}\right) +\left( 2u_{s}+\frac{1}{\underline{\sigma }^{2}}L^{2}\right) \left\vert \hat{Y}_{s}^{\alpha ,\beta }\right\vert ^{2}+\underline{\sigma }^{2}\left\vert \hat{Z}_{s}^{\alpha ,\beta }\right\vert^{2}.
	\end{align*}%
	Thus, 
	\begin{align*}
		&\left\vert \hat{Y}_{t}^{\alpha ,\beta }\right\vert
		^{2}+\int_{t}^{T}\left( 2\theta _{s}-2u_{s}-\frac{1}{\underline{\sigma }^{2}}L^{2}\right)\left\vert \hat{Y}_{s}^{\alpha,\beta }\right\vert^{2}e^{2\int_{0}^{s}\theta_{r}dr}ds+\int_{t}^{T}\left\vert \hat{Z}_{s}^{\alpha ,\beta }\right\vert^{2}e^{2\int_{0}^{s}\theta _{r}dr}\left( d\langle B\rangle s-\underline{\sigma }^{2}ds\right)  \\
		\leq &\left( \alpha +\beta \right) C_{L}\int_{t}^{T}\left( \left\vert Z_{s}^{\alpha }\right\vert ^{2}+\left\vert h(s)\right\vert^{2}+u_{s}^{2}\left\vert Y_{s}^{\alpha }\right\vert ^{2}+u_{s}^{2}\left\vert Y_{s}^{\beta }\right\vert ^{2}\right) e^{2\int_{0}^{s}\theta_{r}dr}ds-\left( N_{T}-N_{t}\right) .
	\end{align*}%
	It turns to%
	\begin{equation*}
		\left\vert \hat{Y}_{t}^{\alpha ,\beta }\right\vert^{2}+N_{T}-N_{t}\leq\left( \alpha +\beta \right) C_{T,L,M,\underline{\sigma }}\int_{t}^{T}\left\vert Z_{s}^{\alpha }\right\vert ^{2}+\left\vert h(s)\right\vert^{2}+u_{s}^{2}\left\vert Y_{s}^{\alpha}\right\vert ^{2}+u_{s}^{2}\left\vert Y_{s}^{\beta }\right\vert ^{2}ds
	\end{equation*}%
	Taking the conditional expectation of both sides, we obtain%
	\begin{equation*}
		\left\vert \hat{Y}_{t}^{\alpha ,\beta }\right\vert ^{2}\leq \left( \alpha+\beta \right) C_{T,L,M,\underline{\sigma }}\hat{\mathbb{E}}_{t}\left[\int_{t}^{T}\left\vert Z_{s}^{\alpha }\right\vert ^{2}+\left\vert h(s)\right\vert ^{2}+u_{s}^{2}\left\vert Y_{s}^{\alpha }\right\vert
		^{2}+u_{s}^{2}\left\vert Y_{s}^{\beta }\right\vert ^{2}ds\right].
	\end{equation*}
	The proof is completed.
\end{proof}

Here, we present the main result of this section.

\begin{theorem}\label{thm solution}
	For some $\lambda>0$, if $\xi\in L_{G}^{2+\lambda}$ and $f$ satisfies (H1)-(H4), then $G$-BSDE (\ref{GBSDE}) exists a unique solution $(Y,Z,K)\in \mathfrak{S}_{G}^{2}(0,T)$.
\end{theorem}

\begin{proof}
	For fixed $\alpha ,\beta>0$, let $f^{\alpha },f^{\beta },\hat{Y}^{\alpha
		,\beta },\hat{Z}_{s}^{\alpha ,\beta }$ be defined as in Lemma \ref{Zcha}.
	
	Step 1. First, let us prove that $\left( Y^{\alpha },Z^{\alpha}\right)_{\alpha >0}$ is a Cauchy sequence. According to Lemma \ref{Ycha}, we yield%
	\begin{equation}
		\hat{\mathbb{E}}\left[ \sup_{t\in \lbrack 0,T]}\left\vert \hat{Y}
		_{t}^{\alpha ,\beta }\right\vert ^{2}\right] \leq \left( \alpha +\beta\right) C_{T,L,M,\underline{\sigma }}\hat{\mathbb{E}}\left[ \sup_{t\in \lbrack0,T]}\hat{\mathbb{E}}_{t}\left[ \int_{0}^{T}\left\vert Z_{s}^{\alpha}\right\vert ^{2}+\left\vert h(s)\right\vert ^{2}+u_{s}^{2}\left\vert Y_{s}^{\alpha }\right\vert ^{2}+u_{s}^{2}\left\vert Y_{s}^{\beta
		}\right\vert ^{2}ds\right] \right] .  \label{EtYcha}
	\end{equation}%
	From Theorem \ref{EsupEt}, we can obtain
	\begin{align*}
		\hat{\mathbb{E}}\left[\sup_{t\in\lbrack0,T]}\hat{\mathbb{E}}_{t}\left[\int_{0}^{T}\left\vert Z_{s}^{\alpha }\right\vert ^{2}ds\right] \right]
		\leq &C_{\lambda ^{\prime }}\left( \hat{\mathbb{E}}\left[ \left(
		\int_{0}^{T}\left\vert Z_{s}^{\alpha }\right\vert ^{2}ds\right) ^{\frac{2+\lambda ^{\prime }}{2}}\right] +1\right) ,
	\end{align*}
	where $0<\lambda ^{\prime }<\lambda $. Similarly, we can get%
	\begin{equation*}
		\hat{\mathbb{E}}\left[ \sup_{t\in \lbrack 0,T]}\hat{\mathbb{E}}_{t}\left[\int_{0}^{T}\left\vert h(s)\right\vert ^{2}ds\right] \right] \leq
		C_{\lambda }\left( \hat{\mathbb{E}}\left[ \left( \int_{0}^{T}\left\vert h(s)\right\vert ^{2}ds\right) ^{\frac{2+\lambda }{2}}\right] +1\right) ,
	\end{equation*}
	and
	\begin{align*}
		\hat{\mathbb{E}}\left[ \sup_{t\in \lbrack 0,T]}\hat{\mathbb{E}}_{t}\left[\int_{0}^{T}u_{s}^{2}\left\vert Y_{s}^{i}\right\vert ^{2}ds\right] \right]
		\leq &M\hat{\mathbb{E}}\left[ \sup_{t\in \lbrack 0,T]}\hat{\mathbb{E}}_{t}\left[ \sup_{s\in \lbrack 0,T]}\left\vert Y_{s}^{i}\right\vert ^{2}\right] \right] \\
		\leq &C_{\lambda ^{\prime },M}\left( \hat{\mathbb{E}}\left[ \sup_{t\in\lbrack 0,T]}\left\vert Y_{t}^{i}\right\vert^{2+\lambda ^{\prime }}\right]+1\right) ,\qquad i=\alpha ,\beta .
	\end{align*}%
	Then by Lemma \ref{ZandK}, (\ref{EtYcha}) turns to 
	\begin{align*}
		&\hat{\mathbb{E}}\left[ \sup_{t\in \lbrack 0,T]}\left\vert \hat{Y}_{t}^{\alpha ,\beta }\right\vert ^{2}\right] \\
		\leq &\left( \alpha +\beta \right) C_{T,\lambda ,\lambda ^{\prime },L,M,\bar{\sigma},\underline{\sigma }}\left( 1+\left\Vert Z^{\alpha }\right\Vert_{H_{G}^{2+\lambda ^{\prime }}}^{2+\lambda ^{\prime }}+\left\Vert Y^{\alpha}\right\Vert _{S_{G}^{2+\lambda ^{\prime} }}^{2+\lambda ^{\prime }}+\left\Vert
		Y^{\beta }\right\Vert _{S_{G}^{2+\lambda ^{\prime} }}^{2+\lambda ^{\prime }}+\hat{\mathbb{E}}\left[ \left( \int_{0}^{T}\left\vert h(s)\right\vert ^{2}ds\right) ^{\frac{2+\lambda }{2}}\right] \right) \\
		\leq &\left( \alpha +\beta \right) C_{T,\lambda ,\lambda ^{\prime },L,M,\bar{\sigma},\underline{\sigma }}\left( 1+\left\Vert Y^{\alpha }\right\Vert_{S_{G}^{2+\lambda^{ \prime} }}^{2+\lambda ^{\prime }}+\left\Vert Y^{\beta}\right\Vert _{S_{G}^{2+\lambda ^{\prime} }}^{2+\lambda ^{\prime }}+\hat{\mathbb{E}}\left[ \int_{0}^{T}\left\vert h(s)\right\vert ^{2+\lambda }ds\right] +\left\Vert \int_{0}^{T}h(s)ds\right\Vert _{L_{G}^{2+\lambda ^{\prime}}}^{2+\lambda ^{\prime} }\right)\\
		\leq &\left( \alpha +\beta \right) C_{T,\lambda ,\lambda ^{\prime },L,M,\bar{\sigma},\underline{\sigma }}\left( 1+\left\Vert Y^{\alpha }\right\Vert_{S_{G}^{2+\lambda^{ \prime} }}^{2+\lambda ^{\prime }}+\left\Vert Y^{\beta}\right\Vert _{S_{G}^{2+\lambda ^{\prime} }}^{2+\lambda ^{\prime }}+\hat{\mathbb{E}}\left[ \int_{0}^{T}\left\vert h(s)\right\vert ^{2+\lambda }ds\right] \right) .
	\end{align*}%
	Due to Theorem \ref{EsupEt} and Lemma \ref{Yalpha}, it follows that 
	\begin{align*}
		\hat{\mathbb{E}}\left[ \sup_{t\in \lbrack 0,T]}\left\vert Y^{i}\right\vert^{2+\lambda ^{\prime }}\right] \leq &C_{T,\lambda ^{\prime} ,L,M,\underline{\sigma }}\left( \hat{\mathbb{E}}\left[ \sup_{t\in \lbrack 0,T]}\hat{\mathbb{E
		}}_{t}\left[ \xi ^{2+\lambda ^{\prime }}\right] \right] +\hat{\mathbb{E}}\left[ \sup_{t\in \lbrack 0,T]}\hat{\mathbb{E}}_{t}\left[ \left(\int_{0}^{T}\left\vert h(s)\right\vert ^{2+\lambda ^{\prime }}ds\right) 
		\right] \right] \right) \\
		\leq &C_{T,\lambda ,\lambda ^{\prime} ,L,M,\underline{\sigma }}\left( 1+\hat{\mathbb{E}}\left[ \xi ^{2+\lambda }\right] +\hat{\mathbb{E}}\left[\int_{0}^{T}\left\vert h(s)\right\vert ^{2+\lambda }ds\right] \right)
		,\qquad i=\alpha ,\beta .
	\end{align*}%
	Therefore, 
	\begin{align*}
		\lim_{\alpha ,\beta \downarrow 0}\hat{\mathbb{E}}\left[ \sup_{t\in \lbrack0,T]}\left\vert \hat{Y}_{t}^{\alpha ,\beta }\right\vert ^{2}\right] 
		\leq&\lim_{\alpha ,\beta \downarrow 0}\left( \alpha +\beta \right) C_{T,\lambda ,\lambda ^{\prime },L,M,\bar{\sigma},\underline{\sigma }}\left( 1+\hat{\mathbb{E}}\left[ \xi ^{2+\lambda }\right] +\hat{\mathbb{E}}\left[ \int_{0}^{T}\left\vert h(s)\right\vert
		^{2+\lambda }ds\right] \right) \\
		=&0.
	\end{align*}%
	It follows that $\left( Y^{\alpha }\right) _{a>0}$ is the Cauchy sequence in $S_{G}^{2}(0,T)$. Thanks to Lemma \ref{Zcha}, we know that%
	\begin{equation*}
		\lim_{\alpha ,\beta \downarrow 0}\Vert \hat{Z}^{\alpha ,\beta}\Vert _{H_{G}^{2}}^{2}=0,
	\end{equation*}
	which implies $\left( Z^{\alpha }\right) _{a>0}$ is the Cauchy sequence in $H_{G}^{2}(0,T)$. Denote the limits of $(Y^{\alpha },Z^{\alpha })$ in $S_{G}^{2}(0,T)\times H_{G}^{2}(0,T)$ as $(Y,Z)$.
	
	Step 2. Now let us prove that 
	\begin{equation*}
		\lim_{\alpha \downarrow 0}\hat{\mathbb{E}}\left[ \left(
		\int_{0}^{T}\left\vert f(s,Y_{s},Z_{s})-f^{\alpha }(s,Y_{s}^{\alpha},Z_{s}^{\alpha })\right\vert ds\right) ^{2}\right] =0.
	\end{equation*}%
	Recall that $\left\vert F^{\alpha }(s,Y_{s},Z_{s})-F^{\alpha
	}(s,Y_{s}^{\alpha },Z_{s})\right\vert =\left\vert F(s,J^{\alpha
	}(s,Y_{s},Z_{s}),Z_{s})-F(s,J^{\alpha }(s,Y_{s}^{\alpha
	},Z_{s}),Z_{s})\right\vert $, and $F$ is continuous in $y$. Then for $\epsilon >0$, there exists $\delta >0$ such that whenever $\left\vert J^{\alpha }(s,Y_{s},Z_{s})-J^{\alpha }(s,Y_{s}^{\alpha },Z_{s})\right\vert<\delta $, it follows that $\left\vert F^{\alpha }(s,Y_{s},Z_{s})-F^{\alpha}(s,Y_{s}^{\alpha },Z_{s})\right\vert <\epsilon $. Set $D_{\alpha }=\left\{ \left\vert J^{\alpha }(s,Y_{s},Z_{s})-J^{\alpha }(s,Y_{s}^{\alpha },Z_{s})\right\vert<\delta \right\} $. From Lemma \ref{Falpha}, we get $D_{\alpha }^{c}=\left\{\left\vert J^{\alpha }(s,Y_{s},Z_{s})-J^{\alpha }(s,Y_{s}^{\alpha},Z_{s})\right\vert \geq \delta \right\} \subset E_{\alpha }=\left\{\left\vert Y_{s}-Y_{s}^{\alpha }\right\vert \geq \delta \right\} $. It leads
	to
	\begin{align*}
		\lim_{\alpha \downarrow 0}\hat{\mathbb{E}}\left[ \int_{0}^{T}I_{E_{\alpha}}ds\right]  \leq &\lim_{\alpha \downarrow 0}\hat{\mathbb{E}}\left[ \frac{1}{\delta ^{2}}\int_{0}^{T}\left\vert Y_{s}-Y_{s}^{\alpha }\right\vert ^{2}ds\right]  \\
		\leq &\frac{T}{\delta ^{2}}\lim_{\alpha \downarrow 0}\hat{\mathbb{E}}\left[\sup_{t\in \lbrack 0,T]}\left\vert Y_{s}-Y_{s}^{\alpha }\right\vert ^{2}\right]  \\
		=&0.
	\end{align*}%
	Clearly, thanks to Lemma \ref{f-falpha}, $F^{\alpha}(\cdot,Y^{\alpha},Z)\in M_{G}^{2}(0,T)$. Then by Lemma 4.2 in \cite{hu2016stochastic}, we know that 
	\begin{equation*}
		\lim_{\alpha \downarrow 0}\hat{\mathbb{E}}\left[ \int_{0}^{T}\left\vert F^{\alpha }(s,Y_{s},Z_{s})-F^{\alpha }(s,Y_{s}^{\alpha },Z_{s})\right\vert^{2}I_{E_{\alpha }}ds\right] < \epsilon ,
	\end{equation*}%
	where $\epsilon $ is the same as above. Thus,
	\begin{align}
		&\lim_{\alpha \downarrow 0}\hat{\mathbb{E}}\left[ \int_{0}^{T}\left\vert f^{\alpha }(s,Y_{s},Z_{s})-f^{\alpha }(s,Y_{s}^{\alpha },Z_{s})\right\vert
		^{2}ds\right]   \notag \\
		\leq &\lim_{\alpha \downarrow 0}\left( 2\hat{\mathbb{E}}\left[
		\int_{0}^{T}\left\vert F^{\alpha }(s,Y_{s},Z_{s})-F^{\alpha}(s,Y_{s}^{\alpha },Z_{s})\right\vert ^{2}ds\right]+2\hat{\mathbb{E}}\left[\int_{0}^{T}u_{s}^{2}\left\vert Y_{s}-Y_{s}^{\alpha }\right\vert ^{2}ds\right] \right)   \notag \\
		\leq &\lim_{\alpha \downarrow 0}\left( 2\hat{\mathbb{E}}\left[
		\int_{0}^{T}\left\vert F^{\alpha }(s,Y_{s},Z_{s})-F^{\alpha
		}(s,Y_{s}^{\alpha },Z_{s})\right\vert ^{2}\left( I_{D_{\alpha
		}}+I_{D_{\alpha }^{c}}\right) ds\right]+2M\hat{\mathbb{E}}\left[ \sup_{t\in\lbrack 0,T]}\left\vert Y_{t}-Y_{t}^{\alpha }\right\vert ^{2}\right] \right) \notag \\
		\leq &2T\epsilon ^{2}+2\lim_{\alpha \downarrow 0}\hat{\mathbb{E}}\left[\int_{0}^{T}\left\vert F^{\alpha }(s,Y_{s},Z_{s})-F^{\alpha}(s,Y_{s}^{\alpha },Z_{s})\right\vert ^{2}I_{E_{\alpha }}ds\right]   \notag\\
		\leq &2T\epsilon ^{2}+2\epsilon .  \label{falphay-falphayalpha}
	\end{align}%
	Finally, by combining Lemma \ref{f-falpha} and (\ref{falphay-falphayalpha}), we can obtain
	\begin{align*}
		&\lim_{\alpha \downarrow 0}\hat{\mathbb{E}}\left[ \int_{0}^{T}\left\vert f(s,Y_{s},Z_{s})-f^{\alpha }(s,Y_{s}^{\alpha },Z_{s}^{\alpha })\right\vert^{2}ds\right]  \\
		\leq &\lim_{\alpha \downarrow 0}\left( 2\hat{\mathbb{E}}\left[
		\int_{0}^{T}\left\vert f(s,Y_{s},Z_{s})-f^{\alpha}(s,Y_{s}^{\alpha},Z_{s})\right\vert ^{2}ds\right] +2\hat{\mathbb{E}}\left[\int_{0}^{T}\left\vert f^{\alpha }(s,Y_{s}^{\alpha },Z_{s})-f^{\alpha}(s,Y_{s}^{\alpha },Z_{s}^{\alpha })^{2}\right\vert ds\right] \right)  \\
		\leq &\lim_{\alpha \downarrow 0}\left( 2\hat{\mathbb{E}}\left[
		\int_{0}^{T}\left\vert f(s,Y_{s},Z_{s})-f^{\alpha }(s,Y_{s}^{\alpha},Z_{s})\right\vert ^{2}ds\right] +2\hat{\mathbb{E}}\left[\int_{0}^{T}L^{2}\left\vert Z_{s}-Z_{s}^{\alpha }\right\vert ^{2}ds\right]\right)  \\
		\leq &\lim_{\alpha \downarrow 0}\left( 2\hat{\mathbb{E}}\left[
		\int_{0}^{T}\left\vert f(s,Y_{s},Z_{s})-f^{\alpha }(s,Y_{s}^{\alpha},Z_{s})\right\vert ^{2}ds\right] +2L^{2}\hat{\mathbb{E}}\left[\int_{0}^{T}\left\vert Z_{s}-Z_{s}^{\alpha }\right\vert ^{2}ds\right]\right)  \\
		\leq &\lim_{\alpha \downarrow 0}\left( 4\hat{\mathbb{E}}\left[
		\int_{0}^{T}\left\vert f(s,Y_{s},Z_{s})-f^{\alpha}(s,Y_{s},Z_{s})\right\vert ^{2}ds\right] +4\hat{\mathbb{E}}\left[\int_{0}^{T}\left\vert f^{\alpha }(s,Y_{s},Z_{s})-f^{\alpha}(s,Y_{s}^{\alpha },Z_{s})\right\vert ^{2}ds\right] \right)  \\
		\leq &8T\epsilon ^{2}+8\epsilon .
	\end{align*}%
	By letting $\epsilon \rightarrow 0$, we conclude the proof. 

	Step 3. Set%
	\begin{equation*}
		K_{t}=Y_{t}-Y_{0}+\int_{0}^{t}f(s,Y_{s},Z_{s})ds-\int_{0}^{t}Z_{s}dB_{s}.
	\end{equation*}%
	It is easy to show that 
	\begin{equation*}
		\lim_{\alpha \downarrow 0}\hat{\mathbb{E}}\left[ \left\vert
		K_{t}-K_{t}^{\alpha }\right\vert ^{2}\right] =0.
	\end{equation*}%
	Therefore, $\left( Y,Z,K\right) $ is the solution to (\ref{GBSDE}).
	
	Step 4. Finally, let us prove the uniqueness of the solution. Consider following equations%
	\begin{align*}
		Y_{t} =&\xi
		+\int_{t}^{T}f(s,Y_{s},Z_{s})ds-\int_{t}^{T}Z_{s}dB_{s}-(K_{T}-K_{t}), \\
		Y_{t}^{\prime } =&\xi +\int_{t}^{T}f(s,Y_{s}^{\prime },Z_{s}^{\prime
		})ds-\int_{t}^{T}Z_{s}^{\prime }dB_{s}-(K_{T}^{\prime }-K_{t}^{\prime }).
	\end{align*}%
	Define $\hat{Y}=Y-Y^{\prime },\hat{Z}=Z-Z^{\prime }$. Applying It\^{o}'s formula, we get 
	\begin{align*}
		&\left\vert \hat{Y}_{t}\right\vert ^{2}+\int_{t}^{T}\left\vert \hat{Z}_{s}\right\vert ^{2}d\langle B\rangle s \\
		=&\int_{t}^{T}2\hat{Y}_{s}\left(f(s,Y_{s},Z_{s})-f(s,Y_{s}^{\prime},Z_{s}^{\prime })\right) ds-\int_{t}^{T}2\hat{Y}_{s}\hat{Z}%
		_{s}dB_{s}-\int_{t}^{T}2\hat{Y}_{s}d\left( K_{s}-K_{s}^{\prime }\right)  \\
		\leq&\int_{t}^{T}2\hat{Y}_{s}\left(f(s,Y_{s},Z_{s})-f(s,Y_{s}^{\prime},Z_{s}^{\prime })\right) ds-\left( N_{T}-N_{t}\right) ,
	\end{align*}%
	where $N_{t}=\int_{0}^{t}2\hat{Y}_{s}\hat{Z}_{s}dB_{s}+\int_{0}^{t}2\hat{Y}
	_{s}^{+}dK_{s}+\int_{0}^{t}2\hat{Y}_{s}^{-}dK_{s}^{\prime }$. Since%
	\begin{align*}
		&2\hat{Y}_{s}\left( f(s,Y_{s},Z_{s})-f(s,Y_{s}^{\prime },Z_{s}^{\prime})\right)  \\
		\leq &2\hat{Y}_{s}\left( f(s,Y_{s},Z_{s})-f(s,Y_{s}^{\prime
		},Z_{s})+f(s,Y_{s}^{\prime},Z_{s})-f(s,Y_{s}^{\prime},Z_{s}^{\prime})\right)  \\
		\leq &2u_{s}\left\vert \hat{Y}_{s}\right\vert ^{2}+2L\left\vert \hat{Y}_{s}\hat{Z}_{s}\right\vert  \\
		\leq &\left( 2u_{s}+\frac{1}{\underline{\sigma}^{2}}L^{2}\right)
		\left\vert \hat{Y}_{s}\right\vert ^{2}+\underline{\sigma}^{2}\left\vert \hat{Z}_{s}\right\vert^{2},
	\end{align*}%
	it follows that%
	\begin{align*}
		&\left\vert \hat{Y}_{t}\right\vert ^{2}+\int_{t}^{T}\left\vert \hat{Z}_{s}\right\vert ^{2}\left( d\langle B\rangle s-\underline{\sigma }^{2}ds\right)  \\
		\leq &\int_{t}^{T}\left( 2u_{s}+\frac{1}{\underline{\sigma }^{2}}L^{2}\right) \left\vert \hat{Y}_{s}\right\vert ^{2}ds-\left(N_{T}-N_{t}\right) .
	\end{align*}%
	Then we yield%
	\begin{equation*}
		\left\vert \hat{Y}_{t}\right\vert ^{2}\leq \hat{\mathbb{E}}_{t}\left[\int_{0}^{T}\left( 2u_{s}+\frac{1}{\underline{\sigma }^{2}}L^{2}\right)
		\left\vert \hat{Y}_{s}\right\vert ^{2}ds\right] .
	\end{equation*}%
	Taking expectations from both sides, we can obtain%
	\begin{align*}
		\hat{\mathbb{E}}\left[ \left\vert \hat{Y}_{t}\right\vert ^{2}\right]  \leq &\hat{\mathbb{E}}\left[ \int_{0}^{T}\left( 2u_{s}+\frac{1}{\underline{\sigma }^{2}}L^{2}\right) \left\vert \hat{Y}_{s}\right\vert ^{2}ds\right]  \\
		\leq &\int_{0}^{T}\left( 2u_{s}+\frac{1}{\underline{\sigma }^{2}}L^{2}\right) \hat{\mathbb{E}}\left[ \left\vert \hat{Y}_{s}\right\vert ^{2}\right] ds.
	\end{align*}%
	From Gronwall's inequality, we know $\hat{\mathbb{E}}\left[ \left\vert \hat{Y}_{t}\right\vert ^{2}\right] =0$, which implies $Y$ is unique in $M_{G}^{2}(0,T)$. Recall that $S_{G}^{2}(0,T)\subset M_{G}^{2}(0,T)$, it follows that $Y$ is unique in $S_{G}^{2}(0,T)$. Then, through a proof method similar to Lemma \ref{Zcha}, we can obtain 
	\begin{equation*}
		\Vert \hat{Z}\Vert _{H_{G}^{2}}^{2}\leq C_{L,M,\bar{\sigma},\underline{\sigma }}\left\{ \Vert \hat{Y}\Vert_{S_{G}^{2}}^{2}+\Vert \hat{Y}\Vert _{S_{G}^{2}}\left( \left\Vert\int_{0}^{T}h(s)ds\right\Vert _{L_{G}^{2}}+\left\Vert Y\right\Vert_{S_{G}^{2}}+\left\Vert Y^{\prime }\right\Vert _{S_{G}^{2}}\right)\right\}.
	\end{equation*}%
	Thus, $\| \hat{Z}\|_{H_{G}^{2}}^{2}=0$. Finally, it can easily be prove that $\hat{\mathbb{E}}\left[ \left\vert K_{t}-K_{t}^{\prime}\right\vert ^{2}\right] =0$.
\end{proof}

Similarly, we can draw the following conclusion.

\begin{theorem}\label{thm2 solution}
	For some $\lambda>0$, if $\xi\in L_{G}^{2+\lambda}$ and $f,g$ satisfies (H1)-(H4), then G-BSDE (\ref{G-BSDE}) exists a unique solution $(Y,Z,K)\in \mathfrak{S}_{G}^{2}(0,T)$.
\end{theorem}

\begin{proof}
	The proof is similar to Theorem \ref{thm solution}, so we omit it.
\end{proof}

\section*{Declarations}

\subsection*{Funding}
Not applicable.
\subsection*{Ethical approval}
Not applicable.
\subsection*{Informed consent}
Not applicable.
\subsection*{Author Contributions}
All authors contributed equally to each part of this work. All authors read and approved the final manuscript.
\subsection*{Data Availability Statement}
Not applicable.
\subsection*{Conflict of Interest}
The authors declare that they have no known competing financial interests or personal relationships that could have appeared to influence the work reported in this paper.
\subsection*{Clinical Trial Number}
Not applicable.

\end{document}